\documentclass{article}
\usepackage[a4paper, left=1in, right=1in, top=1in, bottom=1in]{geometry}

\usepackage{lipsum}
\usepackage{amsfonts}
\usepackage{graphicx}
\usepackage{epstopdf}
\usepackage{algorithmic}
\usepackage{amsmath}	
\usepackage{amsthm}
\usepackage{amscd}
\usepackage{amsfonts}
\usepackage{amssymb}		
\usepackage{mathtools}	
\usepackage{mathrsfs}	
\usepackage[english]{babel}		
\usepackage[utf8]{inputenc}	
\usepackage{calc}
\usepackage{caption}
\usepackage{thumbpdf}			
\usepackage{latexsym}
\usepackage{xcolor}
\usepackage{enumerate}		
\usepackage{url}				
\usepackage{cite}	
\usepackage{graphics} 
\usepackage{epsfig} 
\usepackage{color}
\usepackage{multirow}
\usepackage{booktabs}
\usepackage{adjustbox}
\usepackage{subcaption}

\usepackage{graphicx} 
\usepackage{color}
\usepackage{comment}
\usepackage{enumitem}
\usepackage{mathtools}
\usepackage{hyperref}

\ifpdf
  \DeclareGraphicsExtensions{.eps,.pdf,.png,.jpg}
\else
  \DeclareGraphicsExtensions{.eps}
\fi

\def\N{\mathbb{N}} 
\newcommand{\R}{\mathbb{R}}
\newcommand{\Rn}{\R^n}
\newcommand{\Rm}{\R^m}
\newcommand{\Rnn}{\R^{n \times n}}

\newcommand{\G}{\mathcal{G}}

\newcommand{\C}{\mathcal{C}}
\renewcommand{\L}{\mathcal{L}}

\DeclareMathOperator{\distG}{\operatorname{dist_\mathcal{G}}} 
\newcommand{\agents}{\lbrace 1, \dots, s \rbrace}
\newcommand{\dataTrain}{\mathcal{D}_{\text{tr}}}
\newcommand{\dataVal}{\mathcal{D}_{\text{val}}}
\newcommand{\dataTest}{\mathcal{D}_{\text{te}}}
\newcommand{\cD}{\mathcal{D}}

\newcommand{\abs}[1]{\left\lvert {#1} \right\rvert}
\newcommand{\norm}[1]{\left\lVert {#1} \right\rVert}
\newtheorem{lemma}{Lemma}[section]
\newtheorem{theorem}[lemma]{Theorem}
\newtheorem{corollary}[lemma]{Corollary}
\newtheorem{definition}{Definition}[section]
\newtheorem{prop}{Proposition}[section]

\newtheorem{remark}{Remark}[section]
\newtheorem{example}{Example}[section]
\newtheorem{setting}{Setting}[section]

\title{Separable Approximations of Optimal Value Functions and Their Representation by Neural Networks\thanks{This research was supported by the Deutsche Forschungsgemeinschaft (DFG, German Research Foundation) - project number 463912816.}}
\author{
    Mario Sperl\thanks{Mathematical Institute, University of Bayreuth, Germany, (\texttt{mario.sperl@uni-bayreuth.de} and \texttt{lars.gruene@uni-bayreuth.de})} \and
    Luca Saluzzi \thanks{Sapienza, University of Rome, Italy, ( \texttt{luca.saluzzi@uniroma1.it})} \and
    Dante Kalise \thanks{Department of Mathematics, Imperial College London, United Kingdom, (\texttt{d.kalise-balza@imperial.ac.uk})} \and 
    Lars Grüne\footnotemark[2]
}

\date{\today}

\begin{document}

\maketitle

\begin{abstract}
The use of separable approximations is proposed to mitigate the curse of dimensionality related to the approximation of high-dimensional value functions in optimal control. The separable approximation exploits intrinsic decaying sensitivity properties of the system, where the influence of a state variable on another diminishes as their spatial, temporal, or graph-based distance grows. This property allows the efficient representation of global functions as a sum of localized contributions. A theoretical framework for constructing separable approximations in the context of optimal control is proposed by leveraging decaying sensitivity in both discrete and continuous time. Results extend prior work on decay properties of solutions to Lyapunov and Riccati equations, offering new insights into polynomial and exponential decay regimes. Connections to neural networks are explored, demonstrating how separable structures enable scalable representations of high-dimensional value functions while preserving computational efficiency.
\end{abstract}



\section{Introduction}

High-dimensional systems are ubiquitous in science and engineering, with applications ranging from physics and biology, to social networks and machine learning. A recurring challenge in these domains is the efficient representation and computation of functions defined over high-dimensional spaces. It is often the case that the computational complexity associated to function representation grows exponentially with the space dimension, a phenomenon referred to as the {\sl curse of dimensionality}. 

One promising approach to mitigating the curse of dimensionality lies in exploiting separability, that is, the ability to approximate a global function as a sum of local contributions involving only a subset of variables. Separability properties can be intrinsic to the problem under study, as in systems where interactions decay with distance or graph connectivity, such as in physical systems governed by locality principles or networked systems with sparse interactions. At a computational level, separability can be enforced by a suitable approximation ansatz, including separation of variables, ANOVA\cite{OW03}, and most recent tools such as tensor decompositions\cite{BM02,OS11}. Theoretical and computational frameworks for separable  approximations offer significant reductions in complexity while retaining essential structural features of the underlying functions.

In this work, we focus on the construction and analysis of separable approximations for functions that exhibit a decaying sensitivity property, where the influence of one variable on another diminishes as their distance (spatial, temporal, or graph-based) increases.

The concept of decaying sensitivity is intrinsically linked to the off-diagonal decay property, a topic extensively explored in the field of numerical linear algebra. A foundational theoretical framework for understanding this behavior is introduced in \cite{benzi2016localization}, where decay bounds are derived for the entries of general matrix functions. This analysis is further advanced in \cite{haber2016}, which examines decay properties in the context of optimal control, specifically for solutions to Lyapunov equations. Additionally, related investigations are presented in \cite{massei2024data}, where the sparsity structure of solutions to Riccati equations is analyzed, providing complementary insights into the decay and sparsity phenomena in matrix equations.
Our particular focus of application are optimal control problems, for which such decaying sensitivity properties in time and space have been intensively studied in the past years, see, e.g., \cite{shin2021controllability,na2020exponential,GrSS20} for decaying sensitivity in time and \cite{shin2022lqr,zhang2023optimal,qu2022scalableReinforcement,gottlich2024perturbations} for decaying sensitivity in space for problems governed by discrete-time dynamics, ODEs and PDEs. See also \cite{shin2022expDecay} for a related property for static nonlinear optimization problems. The key prerequisite for the occurrence of decaying sensitivity in space is that the optimal control problem consists of subsystems that are interconnected by a graph structure. For instance in multi-agent control systems or robotic swarms, each agent's state and control variables are influenced predominantly by neighboring agents, leading to a natural decaying sensitivity structure. In \cite{DKK21}, it was shown that low-rank approximations to linear-quadratic value functions explicitly depend on the connectivity of the dynamics. By leveraging this decaying sensitivity property, we develop computationally efficient methods for approximating optimal value functions through separable structures. 

\subsection{Contribution}

The main purpose of this paper is to show that decaying sensitivity in space enables the construction of accurate separable approximations. While our result is formulated for arbitrary functions and thus provides a general framework for separable approximations, we focus on the approximation of value functions arising in the context of dynamic programming or reinforcement learning algorithms. Preliminary versions of the results in this paper appeared in the conference paper \cite{sperl2023separable}, however, the estimates provided in this paper significantly improve upon the results in \cite{sperl2023separable}. In particular, we also address the case in which the sensitivity does not decay exponentially but only polynomially. Moreover, we comment on the relation of our result to known approximation-theoretic methods, such as {\sl anchored decompositions}.

Furthermore, we establish connections between the existence of separable approximations and neural networks, illustrating how structured architectures can benefit from separability to achieve scalable representations of high-dimensional functions. Through theoretical results and numerical experiments, we demonstrate the potential of this approach in reducing computational complexity while maintaining high approximation accuracy.

\subsection{Outline}

After defining the setting of this paper in Section \ref{sec:setting}, our main results are formulated and proved in Section \ref{sec:sep}. Here, we also explain the relation of our construction to the \emph{anchored decomposition}. In Section \ref{sec.ExistenceDecay} we recall results on the existence of exponentially decaying sensitivity for discrete-time linear quadratic optimal control problems and provide a decaying sensitivity result for a class of continuous-time linear-quadratic optimal control problems. In Section \ref{sec:NN} we explain why the existence of separable approximations reduces the complexity of representing high-dimensional functions with neural networks and in Section \ref{sec:num} we provide numerical examples that illustrate these findings.

\subsection{Notation}
Throughout this paper, we adopt the following notation. For a vector \( x \in \Rn \), we denote its one-dimensional components as \( x_j \) for \( 1 \leq j \leq n \). The entries of a given matrix \( A \), which may include block-entries, are denoted by \( A[i,j] \). The identity matrix in \( \mathbb{R}^{n \times n} \) is denoted by \( I_n \), and the null matrix of the same dimension is represented by \( 0_n \). Additionally, unless otherwise specified, the notation \( \| \cdot \| \) refers to the Euclidean 2-norm. The comparison space \( \mathcal{L} \) consists of all continuous and strictly decreasing functions \( \gamma \colon \mathbb{R}_{\geq 0} \to \mathbb{R}_{\geq 0} \) that satisfy the condition \( \lim_{\tau \to \infty} \gamma(\tau) = 0 \).  For a set $\Omega \subset \Rn$ and a function $f \colon \Omega \to \Rm$, we set
$
    \lVert f \rVert_{\infty, \Omega} \coloneqq \sup_{x \in \Omega} \lVert f(x) \rVert.
$
If $f$ is continuously differentiable, we further define 
\begin{equation*}
    \| f \|_{\mathcal{C}^1(\Omega)} \coloneqq \| f \|_{\infty, \Omega} + \| \nabla f \|_{\infty, \Omega}. 
\end{equation*}

\section{General setting for decaying sensitivity} \label{sec:setting}

In this section, we introduce sensitivity for graph-based problems through a general framework. Though our primary results in Section \ref{sec.ExistenceDecay} address decaying sensitivity in optimal control problems, the underlying concepts have broader applicability. This general formulation is valuable because the existence of separable approximations -- a key finding related to decaying sensitivity -- extends beyond optimal control to other graph-based problems.

\begin{setting} \label{setting:general}
  Let us consider a system of $s \in \mathbb{N}$ agents with states $z_1, \dots, z_s$ with $z_j \in \mathbb{R}^{n_j}$ for $j = 1, \dots, s$, which together constitute a vector in the state space $\mathbb{R}^n$, where $n = \sum_{j=1}^s n_j$. Denote by $x = (z_1, \dots, z_s) \in \mathbb{R}^n$ the combined state of all agents. The interactions among the agents are represented by a directed graph $\mathcal{G} = (\mathcal{V}, \mathcal{E})$, consisting of $s$ vertices $\mathcal{V} = \agents$, corresponding to the agents, with an edge $(i,j) \in \mathcal{E}$ indicating that agent $z_i$ exerts some influence on $z_j$. For $i, j \in \agents$, define $\text{dist}_{\mathcal{G}}(i,j)$ to be the length of the shortest path from node $i$ to node $j$ in $\mathcal{G}$ (possibly $\infty$ if no path exists and $0$ if $i = j$). We are interested in computing a quantity of interest given by a function $V \colon \Omega \subset \mathbb{R}^n \to \mathbb{R}$. 
\end{setting}
Note that we use a directed graph $\mathcal{G}$ to describe the interaction among agents. However, in the special case of symmetric influence, these interactions can be represented using an undirected graph, which can be viewed as a particular case of a directed graph. While the discussion in Section \ref{sec.ExistenceDecay} and the presented examples focus on the symmetric case, leading to undirected graphs, we adopt the more general formulation with directed graphs, as the main results in Section 3 hold for both cases. We emphasize that the edges in our graph are defined as tuples of nodes without assigned weights and the length of a path is determined by the number of its edges. An interesting extension would be to consider weighted graphs, where the weight of an edge represents the strength of the coupling between the nodes and the graph distance between two nodes is defined as the sum of the weights along a path. A detailed analysis of this setting is, however, deferred to future research.

In the following, we formalize the concept of a decaying sensitivity between agents with an increasing graph distance. To this end, for each $j \in \agents$ define the mapping 
    \begin{equation} \label{eq:setting:def_gj}
        g_j \colon \Rn \to \R , \qquad 
        x = (z_1, \dots, z_s) \mapsto V(x) - V(z_1, \dots, z_{j-1}, 0, z_{j+1}, \dots, z_s).
    \end{equation}
Each such mapping $g_j$ allows us to quantify the influence that the value of $z_j$ has on the evaluation of $V$ at some point $x$ in the state space compared to the anchor $0$. While the mappings $g_j$ above could be defined using a general anchor in place of $0$ in the second term of the difference, we choose to compare with $0$ to avoid notational overload. We are now interested in how much the value of such a function $g_j$ at some point in the state space is affected by a change of the values $z_i$ with $i \in \agents, i \neq j$. We characterize this effect by examining the Lipschitz constant of the mapping $g_j$ with respect to $z_i$. We call $g_j$ Lipschitz with respect to $z_i$ at a point $x = (z_1, \dots, z_s) \in \Omega$ with Lipschitz constant $L > 0$ if for all $\tilde{x} = (z_1, \dots, z_{i-1}, \tilde{z}_i, z_{i+1}, \dots, z_s) \in \Omega$ it holds that 
\begin{equation*}
    |g_j(x) - g_j(\tilde{x})| \leq L \| z_i - \tilde{z}_i \|. 
\end{equation*}
\begin{definition} \label{def:setting:sensitivity}
   We say that $V$ has the $\gamma$-decaying sensitivity property with respect to $\G$ on $\Omega$ for $\gamma \in \L$ if for all $i, j \in \agents$ with $i \neq j$ and for all $x \in \Omega$ the mapping $g_j$ as defined in \eqref{eq:setting:def_gj} is Lipschitz w.r.t. $z_i$ at the point $x = (z_1, \dots, z_s)$ with Lipschitz constant
    \begin{equation} \label{eq:def_Lipschitz}
        L_{i,j} \coloneqq \gamma (\distG(i,j)) \norm{z_j}. 
    \end{equation}
\end{definition}
If the considered graph $\mathcal{G}$ and the set $\Omega$ are clear from the context, we refer to this property as $\gamma$-decaying sensitivity. Intuitively, Definition \ref{def:setting:sensitivity} implies that the influence an agent $z_i$ has on $z_j$, as measured by the mapping $g_j$, decreases as the graph distance between $i$ and $j$ increases. Note that in the situation of Definition \ref{def:setting:sensitivity} we may have $\distG(i,j) = \infty$ for some $i,j \in \agents$ if there is no path from node $i$ to node $j$ in $\mathcal{G}$. In this case, \eqref{eq:def_Lipschitz} is to be understood as $L_{i,j} = 0$, thus stating that $g_j$ is independent of $z_i$.

\section{Separable approximation through decaying sensitivity}\label{sec:sep}
In this section, we first clarify the relationship between the decaying sensitivity property introduced in Definition \ref{def:setting:sensitivity} and the Hessian of the function $V$, under the assumption that $V$ is twice continuously differentiable. We then revisit the approximation method for $V$ proposed in \cite{sperl2023separable} and apply Definition \ref{def:setting:sensitivity} to establish an enhanced error bound that is independent of the number of agents $s$ for settings exhibiting either exponential or polynomial decay in sensitivity. Finally, we highlight a connection between our construction and the anchored decomposition, a well-established concept in approximation theory. Throughout this section, we work within the framework established in Setting \ref{setting:general}.

\subsection{Relation between decaying sensitivity and the Hessian}
Assuming twice differentiability of the function $V$, we can deduce $\gamma$-decaying sensitivity from a corresponding property of its Hessian, as established in the following result.
\begin{prop} \label{prop:Decay_Hessian}
    Let $\Omega \subset \Rn$ be convex and assume that $V$ is $\C^2(\Omega)$. If there exist $\gamma \in \L$ such that for all $i,j \in \agents$ with $i \neq j$ and for all $x = (z_1, \dots, z_s) \in \Omega$ we have 
    \begin{equation*}
      \norm{\frac{\partial^2}{\partial z_i \partial z_j} V(x)} \leq \gamma(\distG(i,j)), 
    \end{equation*}
    then $V$ has $\gamma$-decaying sensitivity on $\Omega$. 
\end{prop}
\begin{proof}
    Fix some $i,j \in \agents$, $i \neq j$. Let $z_{-j} = (z_1, \dots, z_{j-1}, z_{j+1}, \dots, z_s) \in \R^{n - n_j}$. By assumption, the mapping 
   \begin{equation*}
       z_j \mapsto \frac{\partial}{\partial z_i} V(z_1, \dots, z_s)
   \end{equation*}
   is continuously differentiable. Thus, for all $z_j \in \R^{n_j}$ with $(z_1, \dots, z_s) \in \Omega$ there exists some $\xi \in [0,1]$ such that 
   \begin{align*}
    & \norm{\frac{\partial}{\partial z_i} g_{j} (z_1, \dots, z_s)} 
    = \norm{ \frac{\partial}{\partial z_i} V(x) - \frac{\partial}{\partial z_i} V(z_1, \dots, z_{j-1}, 0, z_{j+1}, \dots, z_s)}  \\ 
    = &  \norm{\frac{\partial^2}{\partial z_i \partial z_j} V(z_1, \dots, z_{j-1}, \xi z_j, z_{j+1}, \dots, z_s) z_j }. 
   \end{align*}
    By assumption and the convexity of $\Omega$ we obtain 
   \begin{equation*}
       \norm{\frac{\partial}{\partial z_i} g_{j}(z_1, \dots, z_s)} \leq \gamma(\distG(i,j)) \norm{z_j}. 
   \end{equation*}
   Thus, $L_{i,j} \coloneqq \gamma(\distG(i,j)) \norm{z_j}$ is a Lipschitz constant of $g_j$ with respect to $z_i$ at the point $x = (z_1, \dots, z_s)$.   Then, applying Definition \ref{def:setting:sensitivity}, the result follows.
\end{proof}

Quadratic functions are of particular interest in optimal control, as they represent solutions to the Linear Quadratic Regulator (LQR) problem, which will be examined in greater detail in Section \ref{sec.ExistenceDecay}. For quadratic functions of the form $V(x) = x^T P x$ for some matrix $P \in \Rnn$ we can characterize the decaying sensitivity of $V$ via a decay in the block structure of $P$. Given an exponential decay, this links Definition \ref{def:setting:sensitivity} to the concept of spatially exponential decaying matrices. In the following, for $P \in \Rnn$ we denote with $P[i,j]$ the block-matrix of $P$ consisting of the columns corresponding to $z_j$ and the rows corresponding to $z_i$.

\begin{definition}[Definition 1 in \cite{zhang2023optimal}]
       A matrix $P \in \Rnn$ is called $(C,\rho)$ spatially exponential decaying (SED) for $C > 0$ and $\rho \in (0,1)$ if for all $i,j \in \agents$  
    \begin{equation*}
        \norm{P[i,j]} \leq C \rho^{\distG(i,j)}.
    \end{equation*} 
\end{definition}

\begin{corollary}
\label{cor:decay_P}
    Let $V(x) = x^T P x$ with $P \in \Rnn$. Assume that there exists $\gamma \in \L$ such that for all $i,j \in \agents$ it holds that 
    \begin{equation}
        \norm{P[i,j]} \leq \gamma(\distG(i,j)). 
    \end{equation}
    Then $V$ admits a $\gamma$-decaying sensitivity property. In particular, if $P$ is $(C, \rho)$ SED for some $C >0$ and $\rho \in (0,1)$, then $V$ admits an exponential decaying sensitivity given by $\gamma(\cdot) = C \rho^{(\cdot)}$. 
\end{corollary}

\begin{proof}
    Since $P$ is the Hessian of $V$, the claim follows directly from Proposition \ref{prop:Decay_Hessian}. 
\end{proof}  

\subsection{Construction of a Separable Approximation}
\label{sec:construction}

In the following, we examine the primary advantage of the decaying sensitivity property, which is that the global behavior of $V$ can be approximated by considering its behavior on local subsets of the graph, leading to a separable approximation of the function $V$.
\begin{definition} \label{def:separable}
Let $n, d \in \N$ with $d \leq n$ and $\Omega \subset \Rn$. A function $F \colon \Rn \to \R$ is called $d$-separable if there exist $k \in \N$ and functions $F_j \colon \R^{d_j} \to \R$ with $d_j \leq d$ for all $1 \leq j \leq k$ such that for all $x \in \Omega$ 
\begin{equation*}
    F(x) =  \sum_{j=1}^k F_j(y_j),
\end{equation*}
where for all $j \in \{ 1, \dots, k \}$ it holds that $y_j = (x_{\kappa_{j,1}}, \dots, x_{\kappa_{j,d_j}})$ for some $\kappa_{j,i} \in \{ 1, \dots, n \}$, $1 \leq i \leq d_j$.
\end{definition}

We consider the following construction for approximating $V$ by a separable function, first presented in \cite{sperl2023separable}. Although the construction therein was developed in the context of an optimal value function $V$, it directly extends to Setting~\ref{setting:general}, that is, to a general function $V$ defined on substates with an associated graph $\mathcal{G}$. The approximation depends on a parameter $l$, which determines the size of local neighborhoods in the graph $\mathcal{G}$ used for the separable representation. For any fixed $l \in \N$ it can be expressed as follows:

\begin{equation} \label{eq:SepApprox:CDC_approx}
    V(x) \approx V(0) + \Psi_l(x) =  V(0) + \sum_{j = 1}^s \Psi_l^j(z_{j,l}), \quad x \in \Omega, 
\end{equation}
where $z_{j,l}$ represents the vector containing the states of subsystem $z_j$ along with all subsystems $z_i$ with $i \geq j$ and $\distG(i,j) \leq l$. Let $d_{j,l} \in \N$ denote the dimension associated with $z_{j,l}$, i.e., $d_{j,l}$ is the sum of the dimensions of the subsystems that have an index larger than $j$ and a graph distance of at most $l$ from $z_j$. Furthermore, let $x_{j,l} \in \Rn$ denote the lifting of $z_{j,l}$ into $\Rn$ by inserting the value $0$ for all subsystems not included in $z_{j,l}$. The component functions $\Psi_l^j$ in \eqref{eq:SepApprox:CDC_approx} are defined as
\begin{equation} \label{eq:SepApprox:DefPsij}
    \Psi_l^j(z_{j,l}) \coloneqq V(x_{j,l}) - V(\Lambda_j x_{j,l}), 
\end{equation}
where $\Lambda_j \colon \Rn \to \Rn$ sets the values of subsystem $j$ to $0$, i.e., for $x = (z_1, \dots, z_s)$ it is 
\begin{equation*}
    \Lambda_j x = (z_1, \dots, z_{j-1}, 0, z_{j+1}, \dots, z_s). 
\end{equation*} 
Figure \ref{fig.NeighborhoodIllustration} illustrates a graph composed of five subsystems, where the graph distance is given by $\text{dist}_\mathcal{G}(i,j) = |i - j|$. According to our notation, in this example we have $z_{1,2} = (z_1, z_2, z_3)$, $x_{1,2} = (z_1, z_2, z_3, 0, 0)$, and $\Psi_2^1(z_{1,2}) = V(z_1, z_2, z_3, 0, 0) - V(0, z_2, z_3, 0, 0)$.
Note that \eqref{eq:SepApprox:CDC_approx} is based on a successive reduction of substates through the functions $\Psi_l^j$; see Figure \ref{fig:coordinates} for an illustration. Specifically, for $x = (z_1, \dots, z_s) \in \Omega$, we begin with the approximation
\[
V(x) \approx \Psi_l^1(z_{1,l}) + V(0, z_2, \dots, z_s).
\]
In the next step, we approximately have
\[
V(0, z_2, \dots, z_s) \approx \Psi_l^2(z_{2,l}) + V(0, 0, z_3, \dots, z_s).
\]
This procedure is continued iteratively until all substates are taken into account. The remaining term corresponds to the value at the origin, which then yields \eqref{eq:SepApprox:CDC_approx}. Note that in general the resulting approximation will depend on the numbering of the subsystems $z_j$.
 
\begin{remark}[Role of the parameter $l$] \label{rem.role_l}
    The expression in \eqref{eq:SepApprox:CDC_approx} constitutes a $d$-separable approximation of $V$ with $d \coloneqq \max_{j \in \agents} d_{j,l}$. The parameter $l$ plays a crucial role
   in this construction. On one hand, it determines the dimension of the subspaces associated with the vectors \( z_{j,l} \): smaller values of \( l \) reduce the number of substates included in \( z_{j,l} \), thereby decreasing \( d_{j,l} \) and consequently lowering the order of separability. This reduction is advantageous, as it results in lower-dimensional functions in the separable approximation of \( V \). On the other hand, larger values of $l$ account for longer graph distances, which generally lead to better approximation quality, see Lemma \ref{lem:SepApprox:QuadraticEstimate} in the following. In conclusion, $l$ provides a trade-off between sparsity and approximation quality. Moreover, note that while the presented construction is valid for all $l \in \mathbb{N}$, choosing $l \geq s$ yields the same graph neighborhoods as $l = s - 1$, since the graph distance between any two nodes in a graph with $s$ nodes is at most $s - 1$.  
\end{remark}

\begin{figure}[h!]
    \centering
    \begin{minipage}[t]{0.45\textwidth}
    \centering
    \includegraphics[scale=0.75]{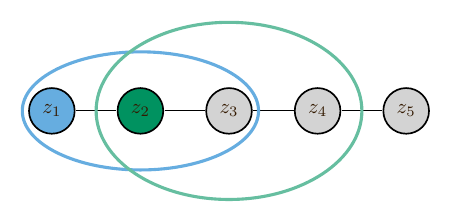}	
    \caption{Graph neighborhoods $z_{1,2}$ (left ellipse) and $z_{2,2}$ (right ellipse) in a sequential graph.}
		\label{fig.NeighborhoodIllustration}
    \end{minipage}%
    \hfill 
    \begin{minipage}[t]{0.45\textwidth}
    \centering
    \includegraphics[width=\linewidth]{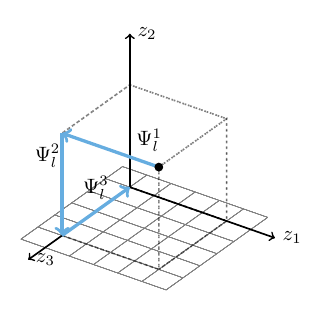}
    \caption{Stepwise isolation of the substates $z_j$ through the functions $\Psi_l^j$ illustrated for $s = 3$.}
    \label{fig:coordinates}
    \end{minipage}
\end{figure}

We now establish an error bound for the approximation in \eqref{eq:SepApprox:CDC_approx} based on the decaying sensitivity property introduced in Definition \ref{def:setting:sensitivity}. In contrast to Theorem 4 in \cite{sperl2023separable},  this error bound is designed to be independent of both the number of agents $s$ and the dimension $n$. However, as a trade-off, the result is restricted to $\L_2$ balls, instead of $\L_\infty$ balls, as presented in \cite[Theorem 4]{sperl2023separable}.  This modification is achieved through the revised concept of decaying sensitivity outlined in Definition \ref{def:setting:sensitivity}, as well as by leveraging the interplay between the type of decay and the growth of the graph as the dimension increases. 
\begin{lemma} \label{lem:SepApprox:QuadraticEstimate}
     Assume that $V$ has the $\gamma$-decaying sensitivity property. Then, for any $l \in \N$ and $x \in \Omega$, we have 
    \begin{equation*}
        \lvert V(x) - \Psi_l(x) - V(0) \rvert \leq \lVert x \rVert_2^2 \lVert D_l \rVert_2, 
    \end{equation*}
    where $D_l \in \R^{s \times s}$ is given by 
    \begin{equation}\label{def.D}
        D_l[i,j] := \begin{cases}
        \gamma({\distG(i,j)}), \quad & \text{if} \, \distG(i,j) > l, \\
       0, \quad & \text{otherwise}. 
    \end{cases}
    \end{equation}
\end{lemma}
\begin{proof}
Initially, we observe that for any $i, j \in \agents$ we have that either $\distG(i,j) \leq s-1 $ or $\distG(i,j) = \infty$. In the latter case it follows that $V(x) - V(\Lambda_j x)$ is independent of $z_i$. Thus, by considering graph-neighborhoods of the size $s-1$ we can express $V(x)$ as 
$V(x) = V(0) + \sum_{j=1}^s \Psi_{s-1}^j(z_{j, s-1})$. This shows the claim for $l \ge s-1$. For the case $l \le s-2$ we use a telescoping sum to write  
\begin{align}       \begin{split}\label{eq:SepApprox:TelescopeV}
        & \Big| {V(x) - \sum_{j=1}^s \Psi_l^j(z_{j,l}) - V(0)}  \Big| =  \Big| {\sum_{j=1}^s \Psi_{s-1}^j(z_{j,s-1}) - \sum_{j=1}^s \Psi_l^j(z_{j,l})}  \Big| \\ \leq & \sum_{j=1}^s \sum_{k = l}^{s-2} \abs{\Psi_{k+1}^j(z_{j,k+1}) - \Psi_{k}^j(z_{j,k})}. 
    \end{split}
\end{align}
Now, fix some $j \in \agents$ and $\tilde{l} \in \{l, \dots, s-2\}$. Let $\tau \in \N$ and $i_1, \dots, i_\tau \in \agents$ such that $\lbrace i \in \agents \mid \distG(j,i) = \tilde l \rbrace = \lbrace i_1, \dots, i_\tau \rbrace $.
Then, by the definitions of $\Psi_l^j$ in \eqref{eq:SepApprox:DefPsij} and $g_j$ in \eqref{eq:setting:def_gj} we have
    \begin{align*}
        & \Psi_{\tilde l}^j(z_{j,{\tilde l}}) -\Psi_{{\tilde l}+1}^j(z_{j,{\tilde l}+1}) = V(x_{j,{\tilde l}}) - V(\Lambda_j x_{j,{\tilde l}}) - V(x_{j,{\tilde l}+1}) + V(\Lambda_j x_{j,{\tilde l}+1}) \\ 
        = &  \sum_{p = 1}^\tau g_j(x_{j, {\tilde l}+1}^{(p)}) - g_j(x_{j, {\tilde l}}^{(p)}), 
    \end{align*}
    where for each $1 \leq p \leq \tau$ the vectors $x_{j, {\tilde l}+1}^{(p)}$ and $x_{j, {\tilde l}}^{(p)}$ are obtained by setting all entries corresponding to $z_{i_1}, \dots, z_{i_p}$ to $0$ in the vectors $x_{j, {\tilde l}+1}$ and $x_{j,{\tilde l}}$, respectively. By successively applying the $\gamma$-decaying sensitivity property as defined in Definition \ref{def:setting:sensitivity}, we obtain
    \begin{equation*}
        \abs{ \Psi_{\tilde l+1}^j(z_{j,\tilde l+1}) -\Psi_{\tilde l}^j(z_{j,\tilde l})} \leq \gamma(\tilde l+1) \sum_{p=1}^\tau \norm{z_{i_p}} \norm{z_j}. 
    \end{equation*}
    From this, using \eqref{eq:SepApprox:TelescopeV} we can derive
    \begin{align*}
         & \Bigg| {V(x) - \sum_{j=1}^s \Psi_l^j(z_{j,l}) - V(0)} \Bigg|  \leq \sum_{j=1}^s \sum_{k = l}^{s-2} \; \sum_{i: \distG(i,j) = k+1} \norm{z_i} \gamma(k+1) \norm{z_j} \\ 
         = &  \sum_{j=1}^s \sum_{i: \distG(i,j) > l} \norm{z_i} \gamma(\distG(i,j)) \norm{z_j}
         =  \sum_{j=1}^s \norm{z_j} \sum_{i = 1}^s D_l[i,j] \norm{z_i} \\
         = & \begin{bmatrix}
            \norm{z_1} \\
            \vdots \\
            \norm{z_s}
            \end{bmatrix}^T D_l \begin{bmatrix}
            \norm{z_1} \\
            \vdots \\
            \norm{z_s}
            \end{bmatrix} \leq  
             \norm{\begin{bmatrix}
            \norm{z_1} \\
            \vdots \\
            \norm{z_s}
            \end{bmatrix}}^2 \lVert D_l \rVert_2 = \lVert x \rVert_2^2 \lVert D_l \rVert_2,  
    \end{align*}
    which shows the claim. 
\end{proof}
In particular, Lemma~\ref{lem:SepApprox:QuadraticEstimate} shows that increasing the graph distance $l$ improves the approximation quality. Specifically, for $l \geq s - 1$, we have $D_l = 0$, which implies $\Psi_l + V(0) = V$. However, in this case, the local state $z_{1,l}$ already contains all substates connected to $z_1$, resulting in a high order of separability. In view of our objective to approximate $V$ using low-dimensional functions, it is therefore not desirable to choose $l \ge s - 1$. Instead, $l$ should be selected as a trade-off between approximation accuracy and numerical efficiency, cf. Remark~\ref{rem.role_l}.

In the following, we specify the error bound derived from Lemma \ref{lem:SepApprox:QuadraticEstimate} for both exponential and polynomial decay scenarios. For this, it is essential to quantify the growth of neighborhoods within the underlying graph.  
We define a function $r \colon \N \to \N$ as a growth bound for the graph $\mathcal{G}$ if it satisfies the condition: 
    \begin{equation} \label{eq.graphBound}
        \forall l \in \agents, j \in \agents:  \; \left\lvert \lbrace i \in \agents \mid \distG(j,i) = l \rbrace \right\rvert \leq r(l).
    \end{equation}
Note that if $r$ is a subexponential function, condition \eqref{eq.graphBound} coincides with Assumption 3.4 in \cite{shin2022lqr}, which is used in Corollary 3.5 therein to show the exponential accuracy of a truncated feedback controller for time-discrete linear quadratic optimal control problems. Further, we require the following lemma concerning induced matrix norms.
\begin{lemma}[Corollary 2.3.2. in \cite{golub2013matrix}] \label{lem.MatrixNormEstimate}
    Let $A \in \R^{s \times s}$. Then it holds that 
    \begin{equation*}
        \lVert A \rVert_2 \leq \sqrt{\lVert A \rVert_1 \lVert A \rVert_\infty}.
    \end{equation*}
\end{lemma}

\begin{theorem}[Exponential Decay]\label{thm.ExpDecay}
Assume that $V$ possesses the $\gamma$-decaying sensitivity property with $\gamma(l) = C \rho^{l}$ for some $C > 0$ and $\rho \in (0,1)$, and that there exists a growth bound $r \colon \N \to \N$ for the underlying graph, satisfying  $r(k) \leq \widehat{C} \mu^k$, $k \in \N$, for some $\widehat{C} > 0$ and $\mu \in [1, \rho^{-1})$. Then for all $l \in \N$ and all $x \in \Omega$ it holds 
    \begin{equation*}
        \lvert V(x) - \Psi_l(x) - V(0) \rvert \leq \widetilde{C} \delta^{l+1} \lVert x \rVert_2^2, 
    \end{equation*}
    where $\widetilde{C} > 0$ and $\delta \in (\rho \mu, 1)$ are independent of $s$ and $n$.  
\end{theorem}
\begin{proof}
Let $l \in \N$, $D_l$ as in \eqref{def.D}, $i \in \agents$, and $\delta \in (\rho \mu, 1)$. Then it holds
\begin{align*}
    & \sum_{j=1}^s \lvert D_l[i,j] \rvert = \sum_{j \in \agents: \distG(i,j) > l} \gamma(\distG(i,j)) \leq \sum_{k = l+1}^\infty \sum_{j : \distG(i,j) = k} \gamma(k) \\ 
    & \leq  C \sum_{k = l+1}^\infty r(k) \rho^k 
    \leq C \sup_{k \in \N} \left\lbrace r(k) \left(\frac{\rho}{\delta}\right)^k \right\rbrace \frac{\delta}{1-\delta} \delta^l. 
\end{align*}
Thus, we obtain 
\begin{align*}
    & \lVert  D_l \rVert_\infty = \max_{i \in [s]} \sum_{j \in [s]} \lvert  D_l[i,j] \rvert 
    \leq C \sup_{k \in \N} \left\lbrace r(k)  \left(\frac{\rho}{\delta}\right)^k \right\rbrace \frac{\delta^{l+1}}{1-\delta} 
        \end{align*}
    \begin{align*}
    \leq C \widehat{C} \sup_{k \in \N} \left\lbrace \mu^k  \left(\frac{\rho}{\delta}\right)^k \right\rbrace \frac{\delta^{l+1}}{1-\delta}  \leq C \widehat{C} \frac{\delta^{l+1}}{1-\delta} \delta^{l+1} =: \widetilde{C}  \delta^{l+1}.
\end{align*}
Analogously, we obtain the same bound for $\norm{D_l}_1$. Finally, using Lemma \ref{lem:SepApprox:QuadraticEstimate} and Lemma \ref{lem.MatrixNormEstimate} yields 
\begin{equation*}
    \lvert V(x) - \Psi_l(x) - V(0) \rvert \leq \lVert x \rVert_2^2 \lVert D_l \rVert_2 \leq \sqrt{\lVert D_l \rVert_1 \lVert D_l \rVert_\infty} \leq \widetilde{C}  \delta^{l+1}. 
\end{equation*}
This shows the claim. 
\end{proof}
Note that the growth bound required in Theorem \ref{thm.ExpDecay} is met if the number of agents with a graph distance of $l$ grows at most polynomial in $l$. If the growth rate is exponential, its base must be less than the reciprocal of the exponential decay base. Similarly to Theorem \ref{thm.ExpDecay}, one can establish a bound on the approximation error assuming a \emph{polynomial} decay in sensitivity. 

\begin{theorem}[Polynomial Decay]\label{thm.PolyDecay}
   Assume that $V$ possesses the $\gamma$-decaying sensitivity property with $\gamma(l) = C (l+1)^{-\alpha}$ for some $C > 0$ and $\alpha > 1$. Furthermore, assume that there exists a growth bound $r \colon \N \to \N$ for the underlying graph and some $\beta < \alpha - 1 $ such that $r \in \mathcal{O}(l^{\beta})$ for $l \to \infty$. Then for all $l \in \N$ it holds that for all $x \in \Omega$
    \begin{equation*}
        \lvert V(x) - \Psi_l(x) - V(0) \rvert \leq \widetilde{C} \lVert x \rVert_2^2 \sum_{k = l+2}^\infty k^{-\alpha + \beta}, 
    \end{equation*}
    where $\widetilde{C} > 0$ is independent of $s$ and $n$.  
\end{theorem}
\begin{proof}
Analogously to the proof of Theorem \ref{thm.ExpDecay}, we obtain 
\begin{align*}
    \lVert  D_l \rVert_\infty \leq  \sum_{k = l+1}^\infty r(k) C (k+1)^{- \alpha} 
    \leq C \sup_{k \in \N} \left\lbrace r(k) (k+1)^{-\beta}  \right\rbrace \sum_{k = l+2}^\infty k^{-\alpha + \beta}. 
\end{align*}
Again, combining Lemma \ref{lem:SepApprox:QuadraticEstimate} and Lemma \ref{lem.MatrixNormEstimate} yields the claim with \linebreak $\widetilde{C} := C \sup_{k \in \N} \left\lbrace r(k) (k+1)^{-\beta}  \right\rbrace$. 
\end{proof}

\begin{remark}
Theorems~\ref{thm.ExpDecay} and~\ref{thm.PolyDecay} are stated for a fixed dimension $n$ and a fixed number of subsystems $s$. However, the resulting approximations are dimension-independent, as the approximation quality does not degrade with increasing $s$ or $n$. Consequently, if we consider a family of problems with increasing dimension and the assumptions are satisfied uniformly with respect to the dimension, i.e., each member of the family admits the same $\gamma$-decaying sensitivity and the same growth-bound $r$, then the resulting approximations are also uniform in the dimension.\end{remark}

\begin{remark}[Anchored decomposition]
The separable approximation is strongly related to the notion of anchored decomposition \cite{kuo2010decompositions,rieger2024approximability}. For a given function $f: \mathbb{R}^n \rightarrow \mathbb{R}$ and an anchor $c  \in \mathbb{R}^n$, the anchored decomposition of the function $f$ with respect to the anchor $c$ is defined as
\begin{equation}
    f(x) = \sum_{ u \subseteq [1,\ldots,d]}f_{u,c}(x_u),
\end{equation}
where $f_{u,c}$ satisfies the following properties
\begin{equation}
f_{\varnothing,c} = f(c), \quad f_{u,c}(x_u) = f((x_u;c)) - \sum_{v \subsetneq u } f_{v,c} (x_v) .
\label{def_ancho}
\end{equation}

Here the notation \((x_u; c)\in\mathbb{R}^d\) denotes the vector obtained by inserting the anchor \(c_i\) into every coordinate \(i\notin u\) and the free variable \(x_i\) into every coordinate \(i\in u\).  Equivalently, if
\[
    y_i \;=\;
    \begin{cases}
      x_i, & i\in u,\\
      c_i, & i\notin u,
    \end{cases}
    \quad i=1,\dots,d,
\]
then
\[
    (x_u; c) \;=\; (y_1,\dots,y_d),
    \qquad
    f\bigl((x_u; c)\bigr) \;=\; f(y_1,\dots,y_d).
\]

The key idea behind anchored decomposition is to express a high-dimensional function $f(z_1,\ldots,z_s)$
 as a sum of terms that depend on a progressively larger number of variables, starting from simpler, lower-dimensional components. This decomposition allows simplifying the computation by focusing on low-dimensional interactions (e.g., pairwise or three-variable interactions) when possible, making high-dimensional problems more tractable. 
The $d$-separable decomposition shares some terms of the anchored decomposition and neglects those which include indices far away according to the graph distance.

For example, let us consider the term $f_{[i,j,k],c}$, which accounts for three elements in the set $\{1,\ldots,s\}$ and assume $d_{\mathcal G}(z_i,z_k) \gg 1$.
We also introduce the notation $f_{k,c}(z_{i_1},\ldots,z_{i_k})$, which is the evaluation of the function $f$ with the variables not in the list fixed at the anchor $c$.
To maintain consistency with the construction outlined in Section \ref{sec:construction}, we fix the anchor point $c$ at the origin and henceforth omit its explicit dependence by writing $f_{[i,j,k]}$ and $f_{k}(x_{i_1},\ldots,x_{i_k})$.
By property \eqref{def_ancho}, we have
\begin{equation}
f_{[i,j,k]} = f_3(z_i,z_j,z_k)-f_{[i,j]} -f_{[j,k]}-f_{[i,k]}-f_{[i]}-f_{[j]}-f_{[k]}-f_{\varnothing}.
\label{eq:step1}
\end{equation}
Again, by definition,
$$
f_{[i_1,i_2]}= f_2(z_{i_1},z_{i_2})- f_{[i_1]}-f_{[i_2]}-f_{\varnothing}, \quad i_1,i_2 \in \{1,\ldots,s\},
$$
$$
f_{[i_1]} = f_1(z_{i_1}) -f_{\varnothing},\quad i_1 \in \{1,\ldots,s\}.
$$
Substituting these into \eqref{eq:step1} and canceling terms, we obtain
$$
f_{[i,j,k]} = f_3(z_i,z_j,z_k)-f_2(z_i,z_j)-f_2(z_j,z_k)-f_2(z_i,z_k)+f_1(z_i)+f_1(z_j)+f_1(z_k)-f_{\varnothing}.
$$
This expression can be further  rewritten in terms of the functions $\{g_j\}_j$ introduced in \eqref{eq:setting:def_gj}, as follows
$$
f_{[i,j,k]} = g_k(z_i,z_j,z_k)-g_k(z_j,z_k)-g_i(z_i,z_k)+g_i(z_i),
$$
where, as before, any variables not explicitly included in the list of entries are fixed at the anchor point.
Now, assume that the function $f$ satisfies the $\gamma$-decaying sensitivity property introduced in Definition \ref{def:setting:sensitivity}. Then, the term $f_{[i,j,k]}$ can be estimated using the Lipschitz continuity of the functions $g_k$ and $g_i$
$$
|f_{[i,j,k]}| \le 2 \gamma(d_{\mathcal G}(z_i,z_k)) \Vert z_k \Vert \Vert z_i \Vert.
$$
We note that the decay of the element $f_{[i,j,k]}$ depends solely on the reciprocal distance and positions of the nodes $z_i$ and $z_k$, rather than on the node $z_j$. 

In summary, anchored decomposition and separable decomposition share a common goal of simplifying high-dimensional functions by focusing on lower-dimensional interactions. 
We note that each separable component \(\Psi_l^j\) can be interpreted as a localized aggregation of the anchored decomposition terms with anchor $c = 0$. Specifically, \(\Psi_l^j\) corresponds to the sum of all functions \(f_{u,0}\) for which the index set \(u\) contains \(j\) and any $i \in u$ satisfies $i \geq j$ and $\distG(i,j) \leq l$. This yields the identity
\[
\Psi_l^j(z_{j,l})
=
\sum_{\substack{
  u\subseteq\{j,j+1,\dots,s\}\\
  j\in u,\;\max_{i\in u}\mathrm{dist}_{\mathcal{G}}(i,j)\le l
}}
f_{u,0}(z_u),
\]
which makes explicit the structural relationship between our separable representation \eqref{eq:SepApprox:CDC_approx} and the anchored decomposition. 
While anchored decomposition systematically includes all terms based on variable subsets, separable decomposition selectively neglects terms involving distant variables, leveraging graph-based distances. This distinction allows separable decomposition to achieve greater computational efficiency by emphasizing localized interactions, making it particularly suited for functions with decaying sensitivity over distance.
\end{remark}

\section{Sensitivity decay in linear-quadratic optimal control} \label{sec.ExistenceDecay}

In this section, we show that the decaying sensitivity property appears naturally in the infinite-horizon Linear Quadratic Regulator (LQR). In Subsection \ref{sec:dislqr} we recall results from \cite{sperl2023separable} (relying on \cite{shin2022lqr}) for discrete-time problems while in Subsection \ref{sec:contlqr} we prove a novel result for continuous-time problems, though under stronger conditions than in the discrete-time case.

Based on Setting \ref{setting:general}, we define the state $x \in \mathbb{R}^n$ as being composed of $s \leq n$ substates $x = (z_1, \dots, z_s)$, where each substate $z_j$ may have an associated control $u_j \in \mathbb{R}^{m_j}$ for $j \in \agents$. Whenever a substate does not have an associated control, we set $m_j = 0$. The collection of all controls is denoted by $u = (u_1, \dots, u_s) \in \mathbb{R}^m$, with $m = \sum_{j=1}^s m_j$. We denote by
$
  \mathcal U_d
  =
  \ell^\infty(\N,\R^m)
$
the space of all bounded control sequences in the discrete‐time case, and by
$
  \mathcal U_c
  =L^\infty([0,\infty),\R^m)
$
the space of all essentially bounded control sequences in the continuous‐time case.

Given matrices $A \in \mathbb{R}^{n \times n}$, $B \in \mathbb{R}^{n \times m}$, $Q \in \mathbb{R}^{n \times n}$, $Q\succeq 0_n$, and $R \in \mathbb{R}^{m \times m}$, $R\succ 0_m$, we consider an infinite-horizon optimal control problem formulated either in discrete time
\begin{align} \begin{split}  \label{eq:LQR_discrete}
		\min_{u \in \mathcal{U}_d} \; & J(x_0, u) = \sum_{k=0}^\infty \ell(x(k), u(k)), \quad   
		\text{s.t.} \; x(k+1) = f(x(k),u(k)), \;  x(0) = x_0, 
\end{split} \end{align}
or in continuous time 
\begin{align} \begin{split}  \label{eq:LQR_continuous}
		\min_{u \in \mathcal{U}_c} \; & J(x_0, u) = \int_0^{\infty} \ell(x(t), u(t)) \,dt\, \quad   
		\text{s.t.} \; \dot{x}(t) = f(x(t),u(t)), \;  x(0) = x_0, 
\end{split} \end{align}
where $\ell(x,u) = x^\top Q x + u^\top R u$, $f(x,u) = Ax + Bu$, and $x_0 \in \Rn$ is an initial value. In both cases, we are concerned with computing the optimal value function 
\begin{equation*}
\begin{aligned}
    V\colon &\mathbb{R}^n \to \mathbb{R}, \\
    &x_0 \mapsto \inf_{u \in \mathcal{U}} J(x_0, u),
\end{aligned}
\end{equation*}
where $\mathcal{U}=\mathcal{U}_d$ in the discrete-time case and $\mathcal{U}=\mathcal{U}_c$ in the continuous-time case.
In the LQR framework, it is well-known that the optimal value function $V$ is quadratic, taking the form $V(x) = x^\top P x$, where $P$ satisfies the Discrete Algebraic Riccati Equation for discrete-time systems, and the Continuous Algebraic Riccati Equation for continuous-time systems. As described in Setting~\ref{setting:general}, the graph $\mathcal{G}$ associated with such a problem is determined by the subsystems and their mutual influences. For the set of vertices $\{1, \dots, s\}$, the graph $\mathcal{G}$ has an edge $(j,i)$ if 
\begin{align}
    \begin{split} \label{eq.conditionEdge}
            & A[i,j] \neq 0 \quad \text{or} \quad \Big(m_j > 0 \text{ and } B[i,j] \neq 0  \Big) \quad \text{or} \quad Q[i,j] \neq 0 \quad \text{or} \\ 
    & \Big(m_i > 0, m_j > 0 \text{ and } R[i,j] \neq 0  \Big), 
    \end{split}
\end{align} 
where $A[i,j] \in \mathbb{R}^{n_i \times n_j}$ denotes the submatrix of $A$ corresponding to the rows of $z_i$ and the columns of $z_j$, $B[i,j] \in \mathbb{R}^{n_i \times m_j}$ denotes the submatrix of $B$ corresponding to the rows of $z_i$ and the columns of $u_j$, and analogously for $Q[i,j]$ and $R[i,j]$. In the case of interactions via control inputs, the conditions \(m_j > 0\) and \(m_i > 0, m_j > 0\) are imposed in \eqref{eq.conditionEdge} prior to the conditions on \(B\) and \(R\), respectively, to accommodate subsystems that do not possess control inputs. Such subsystems can only influence other subsystems through the matrices $A$ or $Q$. 

Note that \eqref{eq.conditionEdge} establishes a connection from  $z_j$ to $z_i$ if either $z_j$ or $u_j$ influences the dynamics of $z_i$, if there is an interaction in the state cost between $z_j$ and $z_i$, represented by $Q$, or if there is an interaction in the control cost between $u_j$ and $u_i$, captured by $R$. While the influence in the costs given by $Q$ and $R$ is always symmetric, leading to bidirectional connections in the graph, this symmetry does not generally hold for the dynamics described by $A$ and $B$. Consequently, $\mathcal{G}$ may still be represented as a directed graph. Yet, the results from \cite{shin2022lqr}, on which we rely in Subsection \ref{sec:dislqr}, use an undirected graph and the results in Subsection \ref{sec:contlqr} require symmetry of $A$. Thus, in the remainder of this section we only consider the case that $\mathcal{G}$ is undirected, which is easily accomplished by adding all edges $(j,i)$ to $\mathcal{G}$ for which $(i,j)$ is contained in $\mathcal{G}$.

As highlighted in Corollary \ref{cor:decay_P}, the decaying sensitivity of the optimal value function is associated with the SED property of the Riccati solution. In the first subsection, we address the discrete-time setting, where we revisit results previously derived in \cite{sperl2023separable}. In the second subsection, we transition to the continuous-time setting and, under appropriate assumptions on the problem data, demonstrate an off-diagonal decay property of the Riccati solution. We introduce some notions that will be used in the following.
\begin{definition}
\label{def:stab}
The pair $(A,B)$ is called stabilizable if there exists
a feedback matrix $K \in \mathbb{R}^{m \times n}$ such that $A -B K$ is stable.
\end{definition}
\begin{definition}
\label{def:detect}
    The pair $(A,C)$ is called detectable if the pair $(A^\top, C^\top)$ is stabilizable.
\end{definition}

\subsection{Discrete-Time Linear Quadratic Regulator}\label{sec:dislqr}

In this section, we examine the decaying sensitivity property in discrete-time linear quadratic systems. To this purpose, we revisit the results on exponential decay in the feedback matrix established in \cite{shin2022lqr} and the decay in the Riccati solution as presented in \cite{sperl2023separable}.
Let us introduce the notion of stability in the discrete setting.

\begin{definition}[Stability in discrete time]
    A matrix $C \in \mathbb{R}^{n \times n}$ is called stable in discrete time (or Schur), if all its eigenvalues lie in the unit circle, $i.e.$ $|\lambda_i| < 1$ $\forall$ $\lambda_i \in \sigma(C)$, where $\sigma(C)$ is the spectrum of the matrix $C$.
\end{definition}
Note that for each $\mu>\max_{\lambda_i\in\sigma(C)}|\lambda_i|$ there is $\kappa>0$ such that
\begin{equation} \|C^k\| \le \kappa \mu^k\label{eq:expdecay}\end{equation}
holds all $k\in\mathbb{N}$.
We therefore consider a family of infinite-horizon discrete optimal control problems expressed in the form \eqref{eq:LQR_discrete}, for which the following conditions hold uniformly, i.e., with the same constants $L,\eta, \kappa >0$ and $\mu\in(0,1)$ for all members of the family:
\begin{enumerate}[label=\roman*)]
\item  $\Vert A \Vert,\Vert B \Vert,\Vert Q \Vert,\Vert R \Vert \le L$,
    \item $(A,B)$ is uniformly stabilizable, i.e., Definition \ref{def:stab} holds with $\Vert K \Vert \le L$ and $C=A-BK$ satisfies \eqref{eq:expdecay}, 
    \item $  Q\succeq 0_n$ and $(A, Q^{\frac{1}{2}})$ is uniformly detectable, i.e., Definition \ref{def:detect} holds with $\Vert K \Vert \le L$ and $C=(A-KQ^{\frac{1}{2}})^T$ satisfies \eqref{eq:expdecay}, 
    \item  $R\succeq \eta I_m$.
\end{enumerate}
The solution of the discrete-time infinite horizon optimal control problem relies on the solution of the Discrete Algebraic Riccati Equation (DARE):
\begin{equation}
P=A^{T}PA-A^{T}PB\left(R+B^{T}PB\right)^{-1}B^{T}PA+Q.
\label{eq_dare}
\end{equation}
It has been proved in \cite{shin2022lqr} that the feedback matrix 
\begin{equation*}
    K = (R+B^\top P B)^{-1} B^T P A
\end{equation*}
exhibits a spatially exponential decaying (SED) property, 
$i.e.$
$$
\|K[i,j]\| \le C_K \rho_K^{\distG(i,j)},
$$
where the constants $C_K$ and $\rho_K$ are defined in Theorem 3.3 of \cite{shin2022lqr}. Starting from this result, in \cite{sperl2023separable} we 
extended the SED property to the matrix $P$, the solution of the DARE, as stated in the following result.

\begin{prop} \label{prop.DecayP}
    Consider a family of  LQR problems of the form \eqref{eq:LQR_discrete} satisfying conditions i)--iv) from above, and assume that for $i,j \in \{1,\ldots,s\}$ we have $dist_{\mathcal G}(i,j) = | i - j |$. Then, the solution $P$ of the DARE \eqref{eq_dare} satisfies
    $$
    \Vert P[i,j] \Vert \le  C_P \rho_P^{|i-j|}
    $$
     for constants $C_P$ and $\rho_P$ independent of $s$.  
\end{prop}

Note that Proposition \ref{prop.DecayP} is stated for the case \(\distG(i,j) = |i-j|\), but its proof extends naturally to any graph where \(| \{ j \mid \distG(i,j) \leq 1 \} |\) remains bounded independently of the dimension for all \( i \in \agents \). For example, if \( B = Q = R = I_n \) and \( A \) is a \( p \)-banded matrix for some \( p \in \mathbb{N} \), it follows that \(| \{ j \mid \distG(i,j) \leq 1 \} | \leq 2p + 1 \) for any dimension and any node \( i \).

\subsection{Continuous-Time Linear Quadratic Regulator}\label{sec:contlqr}

In this section, we show that a similar property also holds for continuous-time LQR problems of the form \eqref{eq:LQR_continuous}. 
Since this is not the main focus of this paper, we do not reproduce the reasoning from \cite{sperl2023separable,shin2022lqr} but rather present a different proof, which is significantly shorter and of interest on its own, but which also requires stricter assumptions.

As previously mentioned, the optimal value function takes the form $ V(x)= x^\top P x$,
where $P \in \R^{n \times n}$ satisfies the Algebraic Riccati Equation (ARE)
\begin{equation}
    A^\top P + P A - P B R^{-1} B^\top P + Q = 0.
    \label{ARE}
\end{equation}
To ensure the existence and uniqueness of a positive semi-definite solution to the ARE, we assume that the system is stabilizable and detectable according to Definition \ref{def:stab} and Definition \ref{def:detect}, respectively, where the notion of stability in the continuous setting is provided in the following definition.
\begin{definition}[Stability in continuous time]
    A matrix $C \in \mathbb{R}^{n \times n}$ is called stable in continuous time (or Hurwitz), if all its eigenvalues lie in the
open left half complex plane, i.e.\ $Re(\lambda_i) <0$ $\forall$ $\lambda_i \in \sigma(C)$, where $Re(\cdot)$ stands for the real part.
\end{definition}

\begin{theorem}
Given a stabilizable pair $(A,B)$ and a detectable pair $(A,Q^{1/2})$, the ARE equation \eqref{ARE} admits a unique stabilizing symmetric positive semidefinite solution $P$.
\label{ex_uni_ARE}
\end{theorem}

For a detailed discussion of this result, we refer to \cite{lancaster1995algebraic}. Throughout this section, the assumptions of Theorem \ref{ex_uni_ARE} will be maintained to ensure the well-posedness of the problem. We now turn to a specific case where the solution to the ARE can be found in closed form.
\begin{prop} \label{prop_sol_ric}
Given a symmetric matrix $A$, and matrices $B =  I_n$, $Q = c I_n$ and $R= \gamma I_n$, with $c,\gamma>0$, the unique positive semi-definite solution of the ARE \eqref{ARE} is given by:
\begin{equation}
P = \gamma \left(\sqrt{A^2+\frac{c}{\gamma} I_n} + A\right).
\label{sol_ric}
\end{equation}

\end{prop}

\begin{proof}

First, observe that $A^2+\frac{c}{\gamma} I_n \succ 0$, ensuring that the matrix square root in the expression for $P$ is well defined. Furthermore, since $\sqrt{\lambda^2 +\frac{c}{\gamma}}> |\lambda|$ and $\gamma>0$, the matrix in \eqref{sol_ric} is positive definite.
Substituting expression \eqref{sol_ric} into the ARE \eqref{ARE} confirms that the expression for $P$ satisfies the equation.
\end{proof}

We now introduce the concept of functions of diagonalizable matrices, which will be instrumental in the subsequent analysis. For further details and alternative definitions, the reader is referred to \cite{higham2008functions}.

\begin{definition}[Function of a diagonalizable matrix]
Let $A \in  \mathbb{R}^{n \times n} $ be a diagonalizable matrix. Then there exists an invertible matrix \( X \in \mathbb{C}^{n \times n} \) and a diagonal matrix  
$
\Lambda = \operatorname{diag}(\lambda_1, \lambda_2, \dots, \lambda_n) \in \mathbb{C}^{n \times n},
$
containing the eigenvalues of \( A \), such that  
$ 
A = X \Lambda X^{-1}.
$
Let \( f(z) \) be a scalar-valued function (real or complex) that is defined at each eigenvalue \( \lambda_i \) of \( A \). The function of the matrix \( A \), denoted \( f(A) \), is then defined by  
\[
f(A) = X f(\Lambda) X^{-1},
\]
where  
\[
f(\Lambda) = \operatorname{diag}(f(\lambda_1), f(\lambda_2), \dots, f(\lambda_n)).
\]
\end{definition}

Then, the expression \eqref{sol_ric} can be reformulated  as $P=f(A)$, where
\begin{equation}\label{eq:riccati-sol-f}
f(x) = \gamma \left(\sqrt{x^2+\frac{c}{\gamma} } + x\right), \quad x \in \mathbb{C}.
\end{equation}
We note that if $A$ is a banded matrix with spectrum in $[a,b]$, the scaled and shifted matrix
$$
\widetilde{A} =  \psi(A) := \frac{2}{b-a} A - \frac{a+b}{b-a} I_n
$$
remains banded with the same bandwidth, and its spectrum is contained in $[-1,1]$. Moreover,  $f(A)= f(\psi^{-1}(\widetilde A))$. Accordingly, we focus on bounding the off-diagonal entries of the matrix function $F(\widetilde A)$, where for $x \in \mathbb{C}$
\begin{equation}
    F(x) := f(\psi^{-1}(x))= \gamma \left(\sqrt{\left( \frac{b-a}{2}x +\frac{a+b}{2}\right)^2+\frac{c}{\gamma} } + \frac{b-a}{2}x +\frac{a+b}{2}\right), 
    \label{f_psi}
\end{equation}
and $\widetilde A$ is banded with spectrum in $[-1,1]$.
The following useful result is adapted from \cite[Theorem 2.2]{benzi1999bounds}.

\begin{theorem}
\label{thm_benzi}
Let $g: \mathbb{C} \rightarrow \mathbb{C}$ be a function analytical in the interior of an ellipse $\mathcal{E}_\chi$ with foci at $\pm 1$ and sum of the half axes $\chi>1$, and continuous on $\mathcal{E}_\chi$. Additionally, let $A\in\mathbb C^{n\times n}$ be symmetric, $p$-banded, and such that $[-1,1]$ contains its spectrum, and let $\rho := \chi^{-\frac{1}{p}}\in(0, 1)$. Then, for all pairs of indices $(i,j)$  such that $i \neq j$, we have that 
$$
|g(A)[i,j]| \le K \rho ^{|i-j|},
$$
where
$$
K = \frac{2 \chi M(\chi)}{\chi-1}, \quad M(\chi) = \max_{z \in \mathcal{E}_\chi} |g(z)|.
$$
\end{theorem}

We now formally state the result on the off-diagonal decay of the solution of the CARE equation \eqref{ARE}, utilizing the previously mentioned theorem and the closed-form solution \eqref{eq:riccati-sol-f}.

\begin{prop} \label{prop.ExistenceDecayCont}
    Let $A \in \mathbb{R}^{n \times n}$ be a $p$-banded symmetric matrix with spectrum in $[a, b]$, and let $f$ be the function in \eqref{eq:riccati-sol-f} that solves the Riccati equation \eqref{ARE} with $B=I_n$, $Q=cI_n$, and $R=\gamma I_n$. Then for all pairs of indices $(i,j)$  such that $i \neq j$, we have that
\begin{equation}
|P[i,j]| = |f(A)[i,j]| \le K \rho^{|i-j|}, \quad i \neq j,
\label{esti_square_root}
\end{equation}
with 
$$
\rho = \chi^{-\frac{1}{p}}, \qquad  \chi = \sqrt{\theta} + \sqrt{\theta+1}, \qquad \theta=\frac{2}{\gamma (b-a)^2}\left(\sqrt{(a^2 \gamma+c)(b^2 \gamma +c)}+ab \gamma +c\right),$$
$$K  = \frac{2 \chi M_F(\chi)}{\chi-1} , \qquad
M_F(\chi) := \max_{z \in \mathcal{E}_{\chi}} |F(z)| = F\left(\sqrt{\theta+1}\right),\qquad F(x)= f(\psi^{-1}(x)),
$$
and $\mathcal{E}_\chi$ is the ellipse with foci at $\pm 1$ and sum of the half axes $\chi>1$.
\label{prop_decay_CARE}
\end{prop}

\begin{proof}
We have $f(A)=f(\psi^{-1}(\widetilde A))=F(\widetilde A)$ and the function $F(x)$ (given in \eqref{f_psi}) satisfies the conditions of Theorem~\ref{thm_benzi}. 
In particular, it is analytical in any ellipse with foci in $-1$ and $1$ that do not include, in its interior part, the points
$$
x_{1/2} = -\frac{b+a}{b-a} \pm i \, 2 \frac{\sqrt{c}}{\sqrt{\gamma}(b-a)}.
$$
Thus, we consider the ellipse $\mathcal{E}_\chi$ with foci in $-1$ and $1$ passing through the points $x_{1/2}$. The sum of the half axes of $\mathcal{E}_\chi$ is given by
$$
\chi = \sqrt{\theta} + \sqrt{\theta+1}, \qquad \theta=\frac{2}{\gamma (b-a)^2}\left(\sqrt{(a^2 \gamma+c)(b^2 \gamma +c)}+ab \gamma +c\right).
$$
Applying Theorem \ref{thm_benzi} yields the desired result. It can be verified that the maximum of $|F(z)|$ on the ellipse occurs at $z = \sqrt{\theta+1}$.
\end{proof}

\begin{example}[Heat equation]
We consider the optimal control problem for the heat equation with homogeneous Neumann boundary conditions, given by:
\begin{equation*}
\partial_t y(t,x) = \sigma \partial_{xx} y(t,x)  + u(t,x),
\label{AC}
\end{equation*}
with $x \in [0,1]$ and $t \in (0,+\infty)$. The corresponding cost functional is defined as:
$$
\tilde{J}(y_0, u) = \int_0^{\infty}  \int_0^1 (|y(t,x)|^2 +  \tilde{\gamma} |u(t,x)|^2) dx \, dt \,.
$$
By discretizing the PDE using finite difference schemes with a spatial stepsize $\Delta x = 5\cdot 10^{-3}$, we approximate the system by the following system of ODEs: 
$$
\dot{y}(s)=A y(s) + Bu(s),
$$
where $A$ is a tridiagonal matrix corresponding to the discretized Laplacian with Neumann boundary conditions.
The discretized cost functional reads
$$
J(y_0,u) = \int_0^{\infty}  \Delta x \, y(t)^\top y(t) +  \gamma  \, u(t)^\top u(t) \, dt \,,
$$
with $\gamma = \tilde{\gamma} \Delta x$ and $n=200$, the number of discretized spatial points. We set $\tilde{\gamma} = 0.01$ and consider two different values for the diffusion coefficient $\sigma$, namely $10^{-3}$ and $10^{-2}$. 

In Figure \ref{fig_decay_bound}, we present a comparison between the decay of the absolute value of the first column of the solution to the CARE \eqref{ARE} and the upper bound derived in Proposition \ref{prop_decay_CARE} for diffusion coefficients $\sigma = 10^{-3}$ (left panel) and  $\sigma = 10^{-2}$ (right panel). In both cases, the bound correctly captures the decay behavior, though a shift is observed. This discrepancy may be attributed to the constant $K$ in \eqref{esti_square_root}, which could potentially be refined through sharper estimates. Additionally, it is evident that increasing the diffusion coefficient results in a slower decay.
Indeed, a higher diffusion leads to a faster communication between nodes and, consequently, to a slower decay of sensitivity.

\begin{figure}[ht]
    \centering
    \begin{subfigure}{0.45\textwidth}
        \centering
        \includegraphics[width=\textwidth]{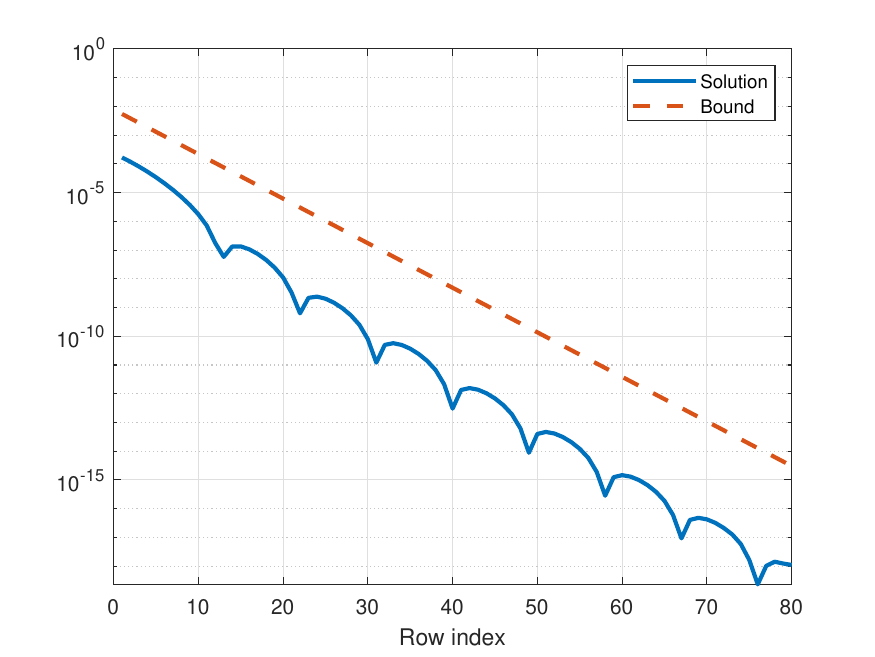}
    \end{subfigure}
    \hfill
    \begin{subfigure}{0.45\textwidth}
        \centering
        \includegraphics[width=\textwidth]{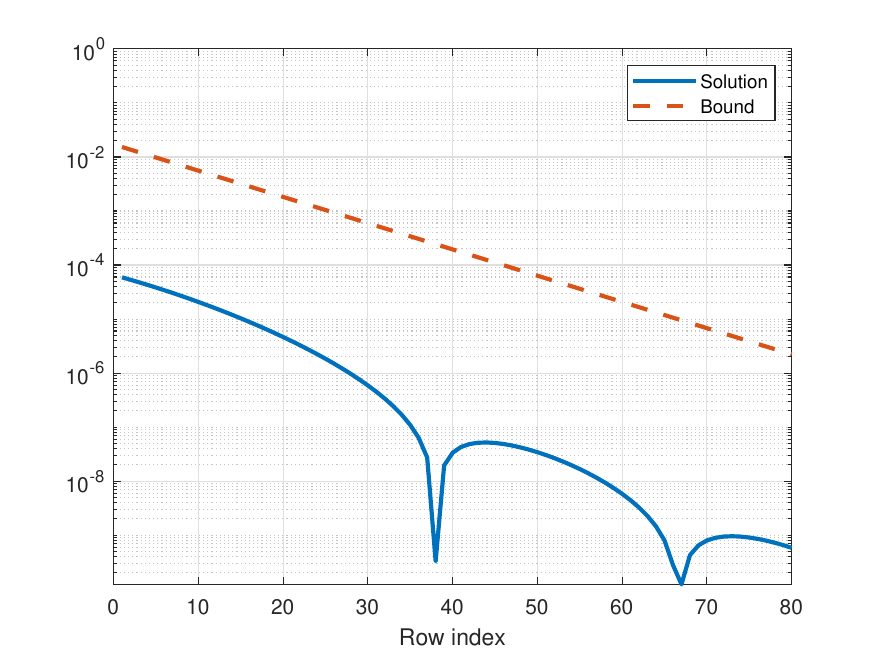}
    \end{subfigure}

	\caption{Decay of the first column's absolute value of the solution of the CARE \eqref{ARE} and the corresponding upper bound \eqref{esti_square_root} with $\sigma = 10^{-3}$ (left) and  $\sigma = 10^{-2}$ (right).} 
 \label{fig_decay_bound}
\end{figure}
\end{example}

\section{Efficient Neural Network Representation for Separable Structures}\label{sec:NN}

In this section, we examine how neural networks can efficiently represent functions exhibiting decaying sensitivity. While neural networks are universal approximators, they typically require an exponential growth in neurons and parameters as input dimensionality increases. However, we demonstrate that for functions $V$ possessing a separable approximation, this curse of dimensionality can be mitigated. Under appropriate conditions, such functions can be represented with neural networks requiring a number of parameters which only grows polynomially. This result is consistent with existing literature showing that certain structural properties enable polynomial parameter scaling with dimension. For example, a similar result holds for compositional functions (cf. \cite{kang2022}) and in particular, for the sub-class of separable functions. The approximation of separable functions with neural networks has been discussed in \cite{gruene2021_ComputingLyapunovFunctions,sperl2024} in the context of (control) Lyapunov functions. While the approach and the neural network construction in the following theorem is similar, there are two major differences compared to these works. The first one is that \cite{gruene2021_ComputingLyapunovFunctions,sperl2024} considers strictly separable functions, meaning that each component $x_i$ of the state $x$ appears in exactly one subsystem. This is not the case in the framework of this paper, since we achieve separability by considering \emph{overlapping} graph neighborhoods, whence we cannot directly apply the results in \cite{gruene2021_ComputingLyapunovFunctions,sperl2024} towards our setting. On the other hand, in \cite{gruene2021_ComputingLyapunovFunctions,sperl2024} the authors used a linear transformation as the first hidden layer for the computation of suitable substates to achieve separability. This is not necessary in our case, since for any given length $l$ of the graph-neighborhoods, the vectors $z_{j,l}$ are uniquely determined by the graph $\mathcal{G}$. 

In the following, a neural network refers to a feedforward network with a one-dimensional output and a linear activation function in its output neuron. We associate each neural network with its corresponding function, denoted by \( W(x; \theta) \colon \mathbb{R}^n \to \mathbb{R} \), where \( x \in \mathbb{R}^n \) represents the input vector and \( \theta \) denotes the trainable parameters.

\begin{theorem} \label{thm:separableNN}
    Let $R > 0$ and consider the setting of Theorem \ref{thm.ExpDecay} with $\Omega := [-R,R]^n$.  Assume that $V$ is continuously differentiable on $\Omega$ with $\lVert V \rVert_{\mathcal{C}^1(\Omega)} \leq C_V$ for some $C_V > 0$. Let $\sigma \in \mathcal{C}^\infty(\R, \R)$ be not a polynomial. Then for any $\varepsilon > 0$ there exists a feedforward neural network $W(x; \theta)$ with activation function $\sigma$ and a corresponding parameter vector $\theta$ such that for any $x \in \Omega$
    \begin{equation*}
        \lvert W(x; \theta) - V(x) \rvert \leq \varepsilon ( 1 + \lVert x \rVert_2^2).
    \end{equation*}
    The number of neurons $N(W)$ in the network $W(x; \theta)$ and its number of trainable parameters $P(W)$ are bounded by
    \begin{equation*}
        N(W) \leq s^{d+1} C^d + n + 1, \quad P(W) \leq s^{d+1} C^d (d + s + 1) + 1, 
    \end{equation*}
    respectively, where $C > 0$ and $d > 0$ are independent of $s$ and $n$, but are determined by $\varepsilon$, $R$, $\widetilde{C}$, $\delta$, and $d_{j,l}$, where $\widetilde{C}$ and $\delta$ are defined as
    in Theorem~\ref{thm.ExpDecay}, and $d_{j,l}$ denotes the dimension of the subvectors $z_{j,l}$ introduced in Subsection~\ref{sec:construction}, with $l$ itself depending on $\varepsilon$, $\widetilde{C}$, and $\delta$.
\end{theorem}
\begin{proof} 
    Let $l \coloneqq \left\lceil{\ln\left(\epsilon {\widetilde{C}}^{-1} \right) \ln^{-1}(\delta) - 1} \right\rceil \in \N$, where $\left\lceil{\cdot }\right\rceil$ stands for the ceiling function. Let $\Psi_l$ be as defined in \eqref{eq:SepApprox:CDC_approx} and \eqref{eq:SepApprox:DefPsij}. Then, by Theorem \ref{thm.ExpDecay}, we obtain that for all $x \in \Omega$
    \[
    |V(x) - \Psi_l(x)| \leq \widetilde{C} \delta^{l+1} \|x\|^2_2 \leq \epsilon \|x\|^2_2.
    \]
    Thus, it is sufficient to show the existence of a neural network that approximates $\Psi_l$ on $\Omega$ with error bounded by $\epsilon$. We begin by constructing subnetworks for approximating $\Psi_l^j$, $1 \leq j \leq s$. Recall that $\Psi_l^j$ depends on $z_{j, l} \in \R^{d_{j, l}}$. Since $\sigma$ is non-polynomial, by Theorem 2.1 in \cite{mhaskar1996_NNApproxSmoothFunctions} and its extension to cubes with length $R$ in \cite{kang2022}, there exists a constant $\mu_{d_{j, l}}$ such that any neural network $W_j(\cdot; \theta_j) \colon \R^{d_{j, l}} \to \R$ consisting of one hidden layer with $M \in \mathbb{N}$ neurons and activation function $\sigma$ satisfies
    \[
    \inf_{\theta_j} |W_j(z_{j, l}; \theta_j) - \Psi_l^j(z_{j, l})|_{\infty, [-R,R]^{d_{j,l}}} \leq \mu_{d_j,l} \, \widetilde{R} \, {M}^{- \frac{1}{d_{j, l}}} \|\Psi_l^j\|_{C^1([-R,R]^{d_{j,l}})},  
    \] 
    where $\widetilde{R} \coloneqq \max\{1, R \}$. Since, by assumption, $V$ has a $C^1$-norm bounded by $C_V$ on $[-R,R]^n$, it follows from \eqref{eq:SepApprox:DefPsij} that $\Psi_l^j$ has $C^1$-norm bounded by $2 C_V$ on $[-R,R]^{d_{j,l}}$, leading to
    \begin{equation} \label{eq:NN:estimate_W_j}
        \inf_{\theta_j} |W_j(z_{j, l}; \theta_j) - \Psi_l^j(z_{j, l})|_{\infty, [-R,R]^{d_{j,l}}} \leq 2  C_V \mu_{d_j,l} \, \widetilde{R} \, {M}^{- \frac{1}{d_{j, l}}}. 
    \end{equation}
 We now construct the neural network \( W(x; \theta) \) as concatenation of the \( s \) subnetworks \( W_1, \dots, W_s \). The network \( W(x; \theta) \) takes an \( n \)-dimensional input vector \( x \), with each subnetwork \( W_j \) then independently processing its corresponding input \( z_{j,l} \), connected to the relevant components of \( x \). The outputs of the subnetworks are aggregated via summation to compute \( W(x; \theta) \). Since each \( W_j \) outputs a linear combination of its \( M \) hidden-layer neurons, these combinations can be unified into a single linear layer connecting the \( sM \) neurons of the subnetworks to the output neuron of \( W(x; \theta) \). By \eqref{eq:NN:estimate_W_j}, the network \( W(x; \theta) \) then satisfies 
    \begin{align*}
        \begin{split} \label{eq:NN:estimate_W}
                    & \inf_{\theta} |W(x; \theta) - \Psi_l(x)|_{\infty, \Omega} \leq \sum_{j=1}^s \inf_{\theta_j} |W_j(z_{j, l}; \theta_j) - \Psi_l^j(z_{j, l})|_{\infty, [-R,R]^{d_{j,l}}} \\ 
                    \leq &  2 s C_V \mu_{d} \, \widetilde{R} \, {M}^{- \frac{1}{d}},  
        \end{split}
    \end{align*}
where $d \coloneqq \max_j d_{j,l}$. Solving for $M$ yields $ M \geq s^{d} (2 C_V \mu_{d} \widetilde{R})^{d} \frac{1}{\varepsilon^{d}}$. 
To establish the existence of a parameter $\theta$ that achieves the desired error bound \( \varepsilon \), we choose 
\begin{equation*}
    M = s^{d} \left\lceil 2 C_V \mu_{d} \widetilde{R} {\varepsilon}^{-1} \right\rceil^d =: s^{d} C^{d} 
\end{equation*}
hidden neurons for each network $W_j$. Hence, the resulting neural network $W$ has at most \( sM = s^{d+1} C^d \) hidden neurons, $n$ input neurons, and one output neuron. Since each subnetwork \( W_j \) contains \( d_{j,l} M \) weights and \( M \) bias terms, and the final linear layer connecting the \( sM \) neurons to the output neuron involves \( sM \) weights and one bias term, the total number of parameters is 
\begin{equation*}
    \sum_{j=1}^s (d_{j,l} M + M) + sM + 1 \leq s^{d+1} C^{d} (d + s + 1) + 1.
\end{equation*}
This completes the proof.
\end{proof}

\begin{figure}
    \centering
    \includegraphics[width=0.6\linewidth]{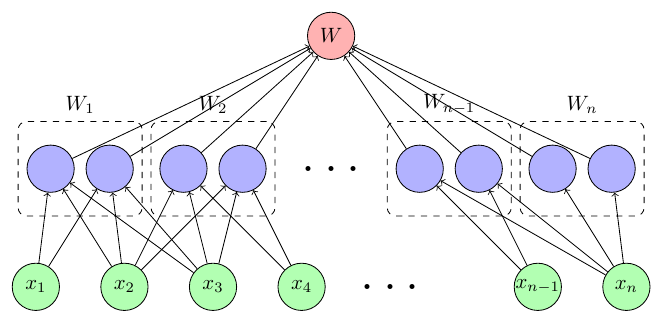}
    \caption{Illustration of the separable structured neural network (S-NN) described in the proof of Theorem \ref{thm:separableNN} for the case of $s = n$ one-dimensional subsystems with $l = 2$ and $\distG(i,j) = \abs{i-j}$.}
    \label{fig:NN_architecture}
\end{figure}

Besides the theoretical value of the result above,  it also provides a systematic approach for designing a neural network with a separable structure, referred to as a Separable-structured Neural Network (S-NN). Figure \ref{fig:NN_architecture} illustrates the architecture of the neural network \( W \) as constructed in the proof of Theorem \ref{thm:separableNN}. 
This example considers the case of one-dimensional subsystems arranged in a sequential graph, where the graph distance is defined as \( \distG(i,j) = |i-j| \), with \( s = n \) and \( z_j = x_j \) for \( 1 \leq j \leq n \). The selected graph-distance parameter is \( l = 2 \), which determines the connectivity of the subnetworks. Specifically, the subnetwork \( W_1 \) is connected to the input nodes \( x_1, x_2 \), and \( x_3 \), while \( W_2 \) is connected to \( x_2, x_3 \), and \( x_4 \). This connectivity pattern continues sequentially, ensuring that subnetwork $W_j$ is connected to input nodes $x_j$, \dots, $x_{j+l}$. The final subnetwork $W_n$ only processes the last subsystem given by $x_n$. \\

Theorem \ref{thm:separableNN} provides an upper bound on the size of a neural network required to represent $V$ under the assumption of exponential decay, as stated in Theorem \ref{thm.ExpDecay} A similar result also holds under the assumption of polynomial decay, as considered in Theorem \ref{thm.PolyDecay}, with the difference that the constants $C$ and $d$ in Theorem \ref{thm:separableNN} are increased. For conciseness, the corresponding result is omitted. 

\begin{remark}
Theorem \ref{thm:separableNN} states that, under the assumption of exponentially decaying sensitivity, neural networks can approximate $V$ with a number of neurons and parameters that grows polynomially with respect to the number of subsystems $s$, where the exponent is $d$. The maximum dimension of a graph-neighborhood $d$ depends on $l$, which is determined by $\varepsilon, \widetilde{C}$, and $\delta$, but independent of $s$ and $n$. Thus, for a family of problems with increasing dimension and uniform decaying sensitivity, where, for any fixed $l$, the maximum neighborhood dimension $d$ remains constant, Theorem \ref{thm:separableNN} describes a situation where the curse of dimensionality is avoided. This applies, for example, to a sequential graph with one-dimensional subsystems, where $d = l + 1$ for any number of subsystems. However, it is important to note that Theorem \ref{thm:separableNN} does not address the required growth of the network in relation to the accuracy \(\varepsilon\).   
\end{remark}

\section{Numerical Experiments: Neural Network Approximations for High-Dimensional Problems} \label{sec:num}

In this section, we provide a detailed analysis of the capability of neural networks to approximate solutions for high-dimensional problems with decaying sensitivity, as stated in Theorem \ref{thm:separableNN}. The first numerical experiment focuses on a LQR example, where the data of the optimal control problem exhibits a banded structure. In the second example, we consider a more general setting, in which the function is no longer quadratic but continues to display sensitivity decay.
The numerical simulations reported in this paper are performed on an Nvidia GeForce RTX 3070. The codes are written in Python 3.9 with TensorFlow 2.10 and are available at \url{https://github.com/MarioSperl/SNN-DecayingSensitivity}.

\subsection{Linear quadratic optimal control problems} \label{subsec:Numerics_LQR}
In this subsection, we first introduce a training algorithm designed to enable a neural network to approximate the optimal value function of linear quadratic optimal control problems exhibiting exponentially decaying sensitivity. Subsequently, we highlight the benefits of the separable network architecture depicted in Figure \ref{fig:NN_architecture} compared to a conventional fully connected network. Finally, we present a numerical experiment that investigates the impact of varying the problem dimension, illustrating how the number of neurons and parameters scales with increasing dimensions in linear quadratic problems.

Consider a continuous-time linear quadratic optimal control problem as described in \eqref{eq:LQR_continuous}. The system matrix $A \in \Rnn$ is randomly generated with a bandwidth of 1, $i.e.$ the matrix is tridiagonal, and entries uniformly distributed in $(0,1)$, using the Python  command \texttt{numpy.random}. Additionally, the matrices $B,Q$ and $R$ are all set to the identity matrix $I_n$. This induces a graph structure as in Figure \ref{fig.NeighborhoodIllustration} with $\distG(i, j) = \abs{i-j}$. From Proposition \ref{prop.ExistenceDecayCont} it follows that the solution $P$ of the corresponding ARE \eqref{ARE} is an SED-matrix, which by Corollary \ref{cor:decay_P} implies that $V(x) = x^T P x $ exhibits an exponentially decaying sensitivity. For dimension $n = 200$, the decay of the first and hundredth row in $P$ are shown in the right column in Figure \ref{fig:V_200}. 
To train the neural network to approximate the value function $V$ on the domain $\Omega = [-1,1]^{200}$, we adopt a supervised learning approach with precomputed, randomly generated, and independently chosen training, validation and test data $\dataTrain, \dataVal, \dataTest \subset \Omega$. These datasets are created by uniformly sampling data points from $\Omega$ and computing the corresponding values of $V$ using the solution $P$ of the ARE. The Mean Squared Error (MSE) on $\cD \in \lbrace \dataTrain, \dataVal, \dataTest \rbrace$ is given as
\begin{equation} \label{eq:NN_ocp:MSE}
    \frac{1}{|\cD|} \sum_{x \in \cD} (V(x) - W(x; \theta))^2. 
\end{equation}
Training is conducted on $\dataTrain$ using a batch-size of $64$, the ADAM optimizer with its default learning rate in TensorFlow, the sigmoid activation function $\sigma(x) = \frac{1}{1 + e^{-x}}$, and TensorFlow’s default parameter initialization. These choices of hyperparameters for training have been applied consistently to all examples presented in this section. Every $10$ epochs, the MSE on the validation dataset $\dataVal$ is evaluated to check whether it falls below a predefined tolerance $\varepsilon$. If this criterion is met or if the maximum number of epochs is reached, the training process is terminated. The final model is then verified by calculating the MSE on the test dataset. 

Particularly in high-dimensional settings, it turned out that incorporating information about the gradient and the values at the origin into the training process can offer notable advantages. To address this, we define a modified loss function $L \colon \Rn \times \R \times \Rn$ as follows:
\begin{align*}
    L(x, W(x; \theta), \nabla W(x; \theta)) \coloneqq & (V(x) - W(x; \theta))^2  + \nu_g \lVert \nabla V(x) - \nabla_x W(x; \theta) \rVert_2^2 \\ & + \nu_z \Big( W(0; \theta)^2 +  \lVert \nabla_x W(0; \theta) \rVert_2^2 \Big). 
\end{align*}
Table \ref{tab:comparison_loss} compares training using the loss function \( \nu_g = \nu_z = 0.5 \) with training based on the standard MSE loss, which corresponds to setting \( \nu_g = \nu_z = 0 \) in $L$. The example is conducted using a fully connected neural network with a single hidden layer comprising 256 neurons across varying dimensions $n$ of the described LQR problem with $A$ having a bandwidth of $1$. The network was trained with $| \dataTrain | = | \dataVal | = 2^{17}$ data points uniformly randomly sampled in $[-1,1]^n$ until the MSE on the validation set as defined in \eqref{eq:NN_ocp:MSE} reached a tolerance of $10^{-3}$. The results demonstrate that incorporating positive values for $\nu_g $ and $ \nu_z $   significantly improves training efficiency, especially in high-dimensional settings. In particular, in the 50-dimensional case, setting $\nu_g = \nu_z = 0.5 $ allows the network to reach the tolerance after 240 epochs, whereas training with $\nu_g = \nu_z = 0$ results in an MSE of $1.3 \times 10^{-2}$ on the validation data after $10^4$ epochs.

\begin{table}[h]
    \centering
    \begin{tabular}{c|cc|cc}
        \toprule
        Dimension & \multicolumn{2}{c|}{$\nu_g = \nu_z = 0.5$} & \multicolumn{2}{c}{$\nu_g = \nu_z = 0$} \\
        \cmidrule(lr){2-3} \cmidrule(lr){4-5}
        & Result & Epochs & Result & Epochs \\
        \midrule
        10  & Tolerance reached & 10 & Tolerance reached & {50} \\
        20  & Tolerance reached & 20 & Tolerance reached & {190} \\
        50  & Tolerance reached & {240}  & Stopped after max. epochs & 1000 \\
        \bottomrule
    \end{tabular}
    \caption{Experimental results for different dimensions and choices of $\nu_g$ and $\nu_z$ in $L$.}
    \label{tab:comparison_loss}
\end{table}

Training with the loss function $L$ requires the computation of the gradient of the neural network with respect to the input, prior to evaluating the gradient of the loss with respect to the network parameters $\theta$. This is achieved using TensorFlow's gradient tape, which enables automatic differentiation by recording operations on tensors during the forward pass and computing gradients during the backward pass. The benefits of such a modified loss function in the context of control systems have also been emphasized in the literature. For example, \cite{azmi2021optimal, kunisch2023optimal, albi2022, dolgov2022data} discuss the inclusion of the information on the derivative, while \cite{kunisch2021semiglobal, breiten2021neural} consider the incorporation of the loss at the origin. Note that while the loss function $L$ is used for training, the mean squared error as defined in \eqref{eq:NN_ocp:MSE} is computed for the validation and test data. \\
We now consider the case $n = 200$ and employ a Separable-structured Neural Network (S-NN) as illustrated in Figure \ref{fig:NN_architecture}, characterized by a graph distance of $l=10$ and a sublayer-size of $M = 16$. This configuration results in $s = 200$ subnetworks, each connected to at most $11$ input nodes, with $16$ nodes in its hidden layer and one output node. Consequently, the total number of trainable parameters, determined by the weights and biases of the neural network, amounts to $40721$. \\ 
Figure \ref{fig:V_200} depicts the optimal value function $V$ and the trained neural network $W(x; \theta)$ projected to the $(x_{1}, x_{2})-$ axis and the $(x_{100}, x_{200})$ axis in the first two columns. The training was performed with $\nu_g = \nu_z = 0.5$, $\abs{\dataTrain} = \abs{\dataVal} = 131072$ points, and has reached a tolerance of $\varepsilon = 10^{-3}$ after {200} epochs. The MSE, computed using $\abs{\dataTest} = 2^{19}$ test points, is ${5.4} \cdot 10^{-4}$.

\begin{figure}[ht]
    \centering
    \begin{subfigure}{0.32\textwidth}
        \centering
        \includegraphics[width=\textwidth]{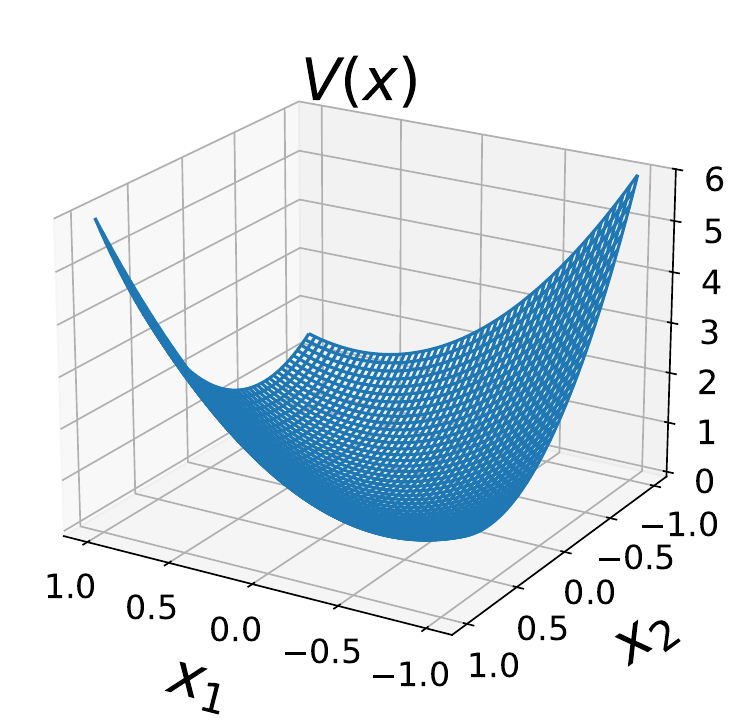}
    \end{subfigure}
    \hfill
    \begin{subfigure}{0.32\textwidth}
        \centering
        \includegraphics[width=\textwidth]{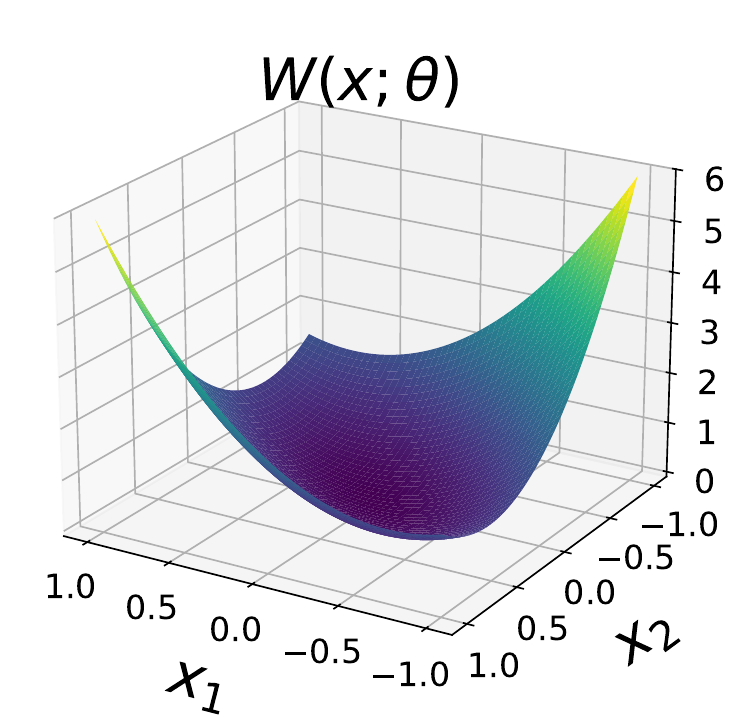}
    \end{subfigure}
    \hfill
    \begin{subfigure}{0.32\textwidth}
        \centering
        \includegraphics[width=\textwidth]{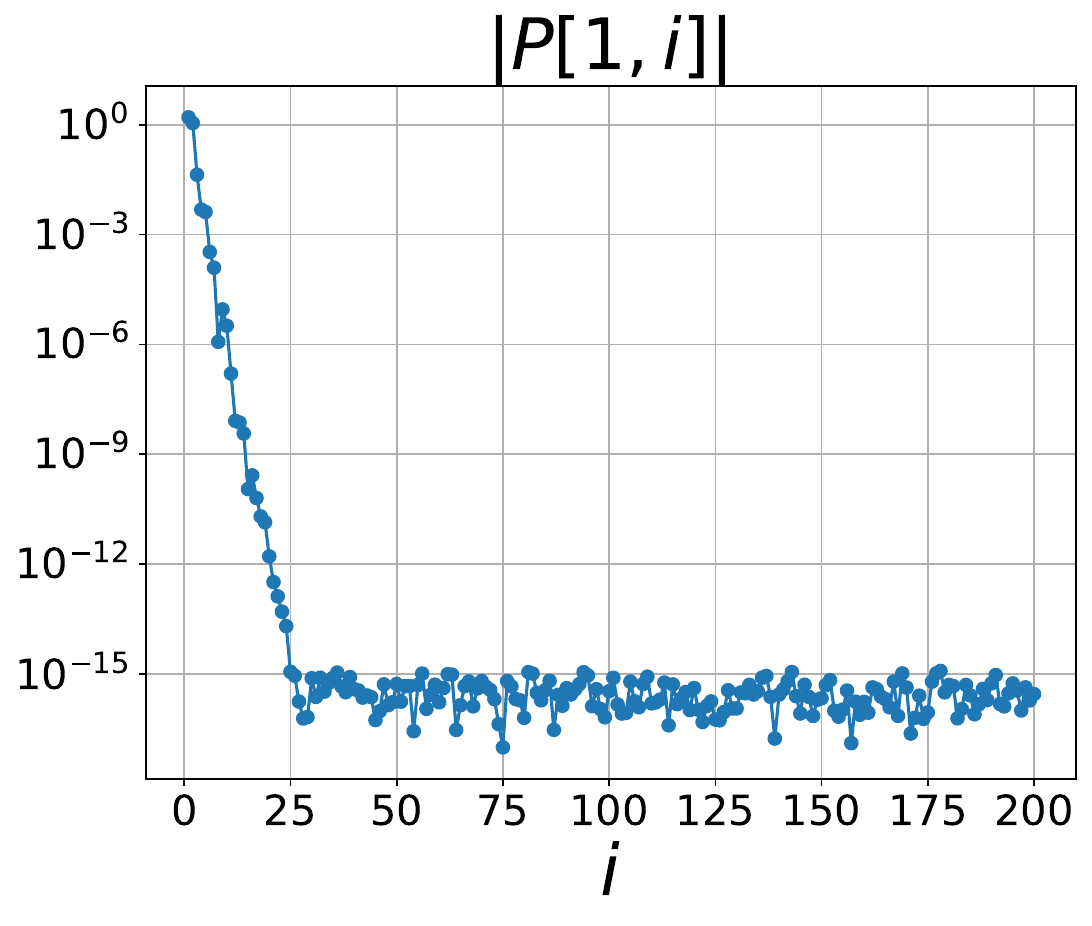}
    \end{subfigure}
    
    \vspace{0.5em} 

    \begin{subfigure}{0.32\textwidth}
        \centering
        \includegraphics[width=\textwidth]{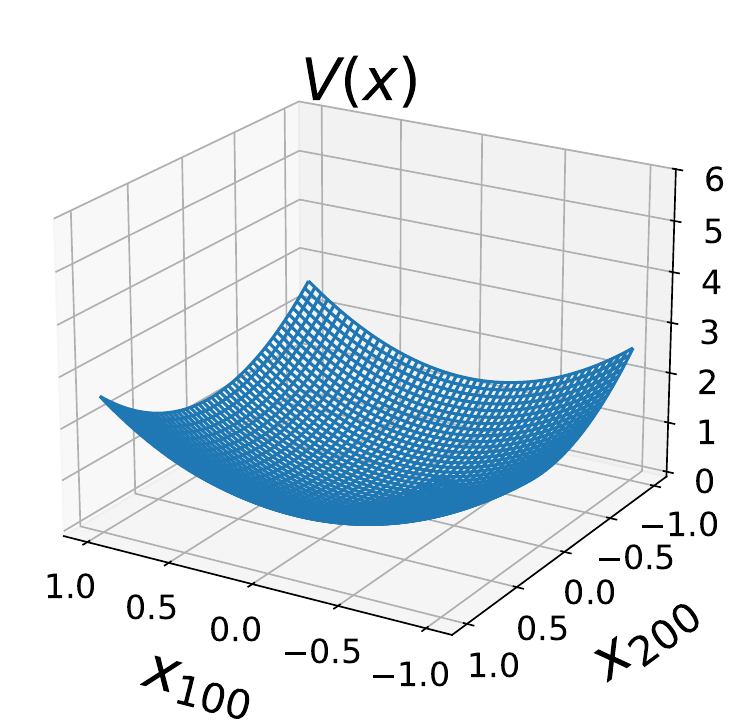}
    \end{subfigure}
    \hfill
    \begin{subfigure}{0.32\textwidth}
        \centering
        \includegraphics[width=\textwidth]{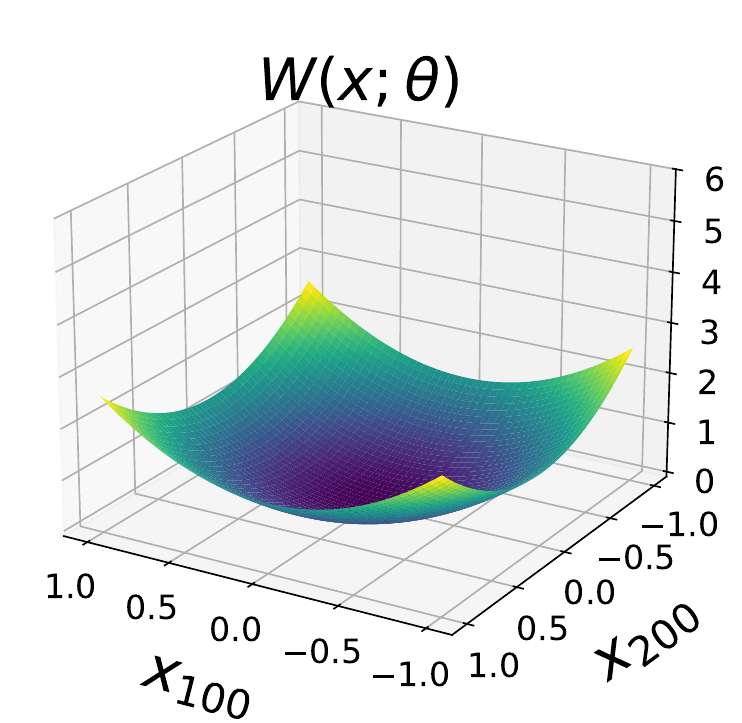}
    \end{subfigure}
    \hfill
    \begin{subfigure}{0.32\textwidth}
        \centering
        \includegraphics[width=\textwidth]{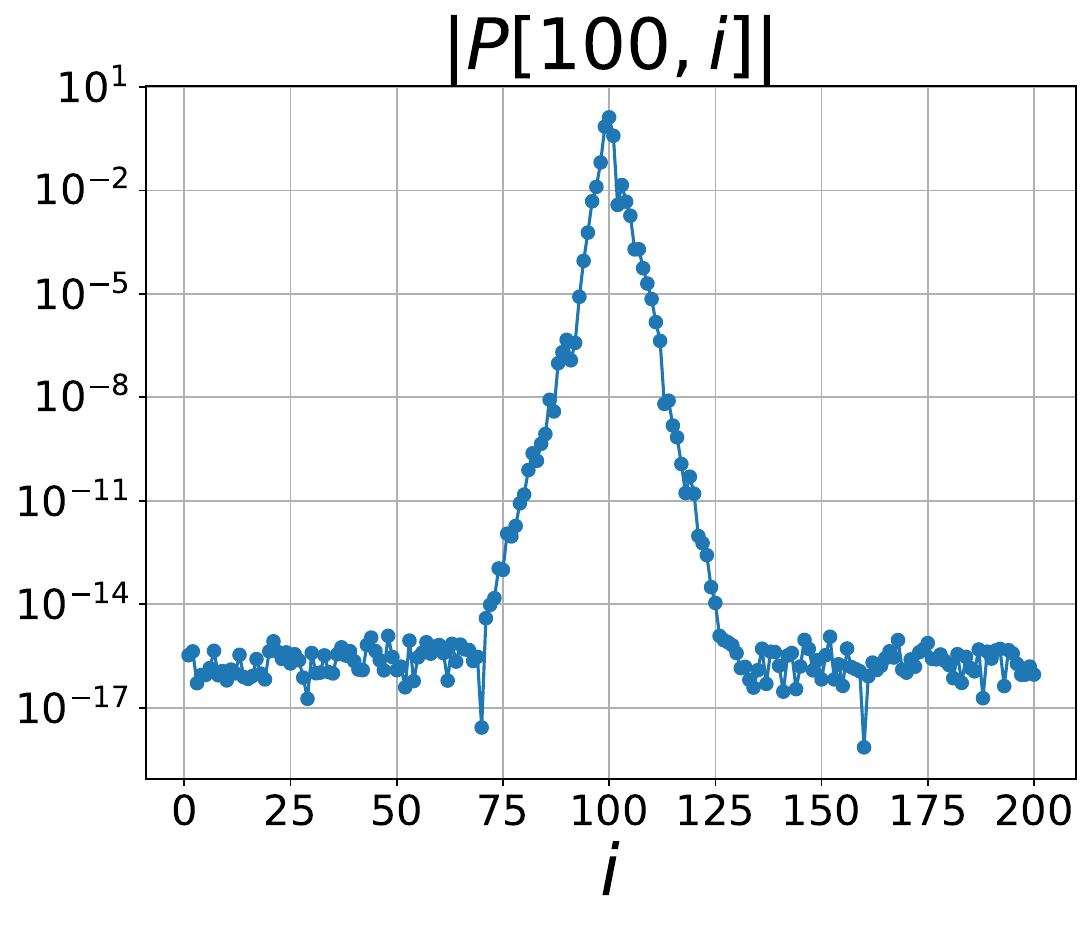}
    \end{subfigure}

    \caption{Plot of the optimal value function $V$ projected onto the $(x_1, x_2)$-axis (upper left) and onto the $(x_{100}, x_{200})$-axis (lower left), the S-NN approximation onto the same hyperplanes (middle column), and the decay in the first row of $P$ (upper right) and the row $100$ of $P$ (lower right).}
    \label{fig:V_200}
\end{figure}

From a theoretical perspective, any function that can be represented using the neural network architecture described in Figure \ref{fig:NN_architecture} can also be represented by a standard feedforward neural network with one hidden layer. However, by explicitly encoding the separable structure within the network architecture, we observe a significant improvement in training performance compared to a fully connected feedforward network. Table \ref{tab:200dim} provides a brief comparison of a fully connected network and a S-NN for the $200$-dimensional LQR problem using the same training parameters as in the experiment corresponding to Figure \ref{fig:V_200}. When the matrix \( A \) is fully connected, lacking any decaying property, the algorithm fails to converge for both network structures, with particularly poor performance observed for the S-NN. However, when \( A \) has bandwidth \( 1 \), introducing an exponential decaying sensitivity property, the fully connected network shows slightly improved performance but still fails to achieve the desired error tolerance within $10^4$ training epochs. In contrast, the S-NN meets the tolerance criterion $\varepsilon = 10^{-3}$ after {$200$} epochs.

\begin{table}[h!]
	\centering
	\begin{tabular}{lccccc}
		\toprule
		\textbf{bandwidth} & \textbf{(sub)-layersize} & \textbf{architecture} & {\textbf{$l$}} & \textbf{test error} & \textbf{epochs} \\
		\midrule
		200  & 2048  & fully conn. & - & $ {1.9} \times 10^0 $ & 1000 \\
		200  & 16 & S-NN & 10 & $8.0 \times 10^3$ & 1000 \\ 
		1  & 2048 & fully conn. & - &  ${3.5} \times 10^{-1}$ & 1000 \\ 
        1  & 16 & S-NN & 10 & ${5.4} \times 10^{-4}$ & {200} \\ 
		\bottomrule
	\end{tabular}
	\caption{A comparison of test errors for a $200$-dimensional example conducted under varying bandwidths, emphasizing the effect of utilizing a separable network structure.}
    \label{tab:200dim}
\end{table}

Another criterion demonstrating the advantage of the S-NN architecture is the reduced number of training points and epochs required to achieve a stopping tolerance of $\varepsilon = 10^{-3}$. Consider a $50$ dimensional example where the matrix $A$ again has a bandwidth of $1$. We compare a fully connected single layer network with $13313$ parameters to a S-NN with $l=10$ and $9521$ parameters. Figure \ref{fig:min_data} presents the results obtained by varying the number of training points $| \dataTrain|$ uniformly randomly chosen in $[-1,1]^{50}$. The maximum number of epochs was adjusted based on the size of the training dataset to ensure the algorithm had sufficient time to train: specifically, it was set to $1000$ for $2^{17}$ training points, $2000$ for $2^{16}$ points, and so on, doubling with each halving of the training set size. As before, the training loss parameters were set to $\nu_g = \nu_z = 0.5$. We observe that, on one hand, the separable structured network requires significantly fewer training points to reach the desired tolerance in the test dataset, see Figure \ref{fig:min_data}. On the other hand, for larger amounts of training data, where both architectures meet the tolerance, the S-NN requires considerably fewer epochs, see Figure \ref{fig:min_epochs}. 

\begin{figure}[h!]
    \centering
    \begin{minipage}[t]{0.45\textwidth}
    \centering
    \includegraphics[width=1\linewidth]{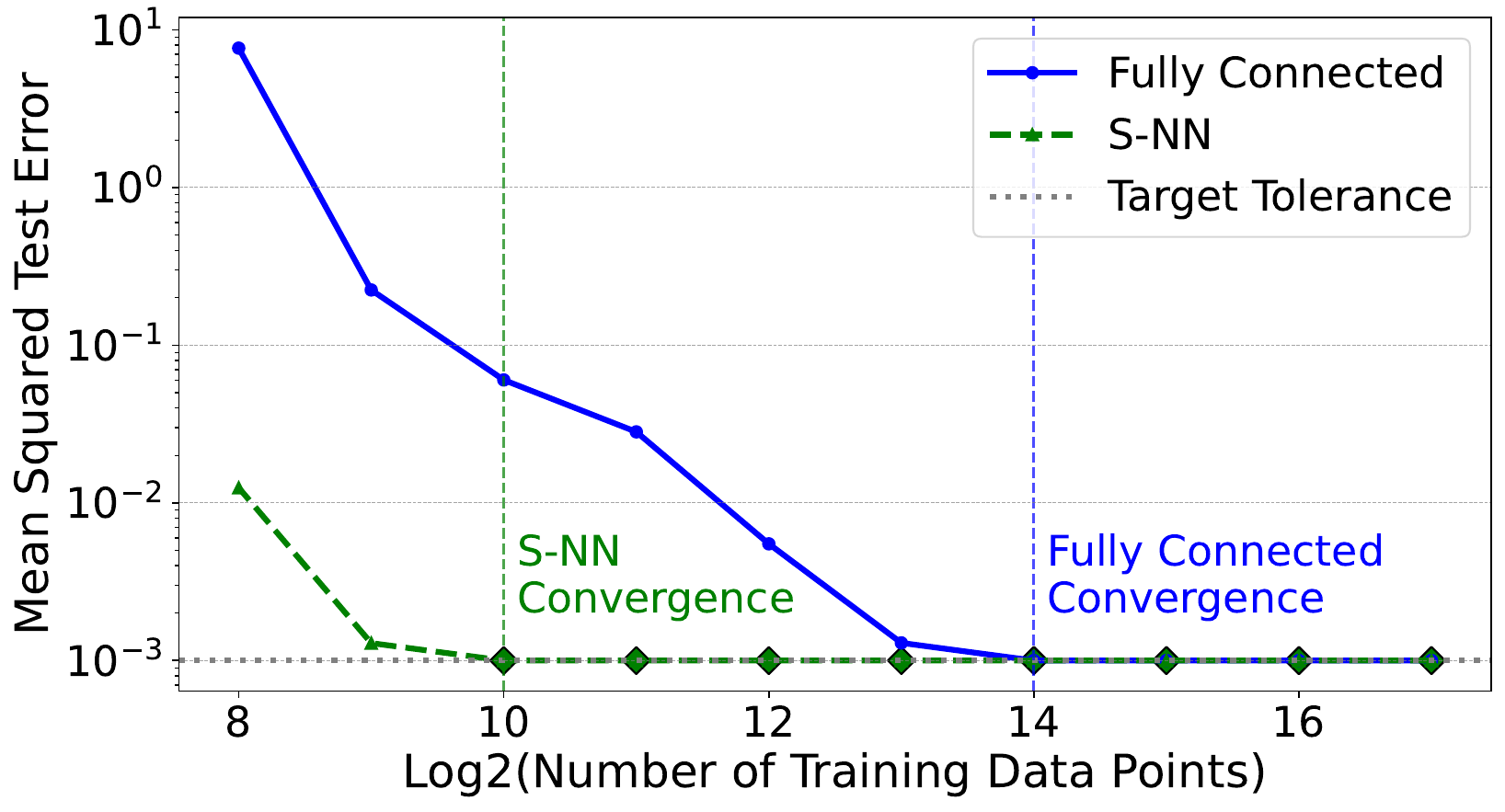}
    \caption{Comparison of test error (log-scale) against the number of training points.}
    \label{fig:min_data}
    \end{minipage}%
    \hfill 
    \begin{minipage}[t]{0.45\textwidth}
    \centering
    \includegraphics[width=\linewidth]{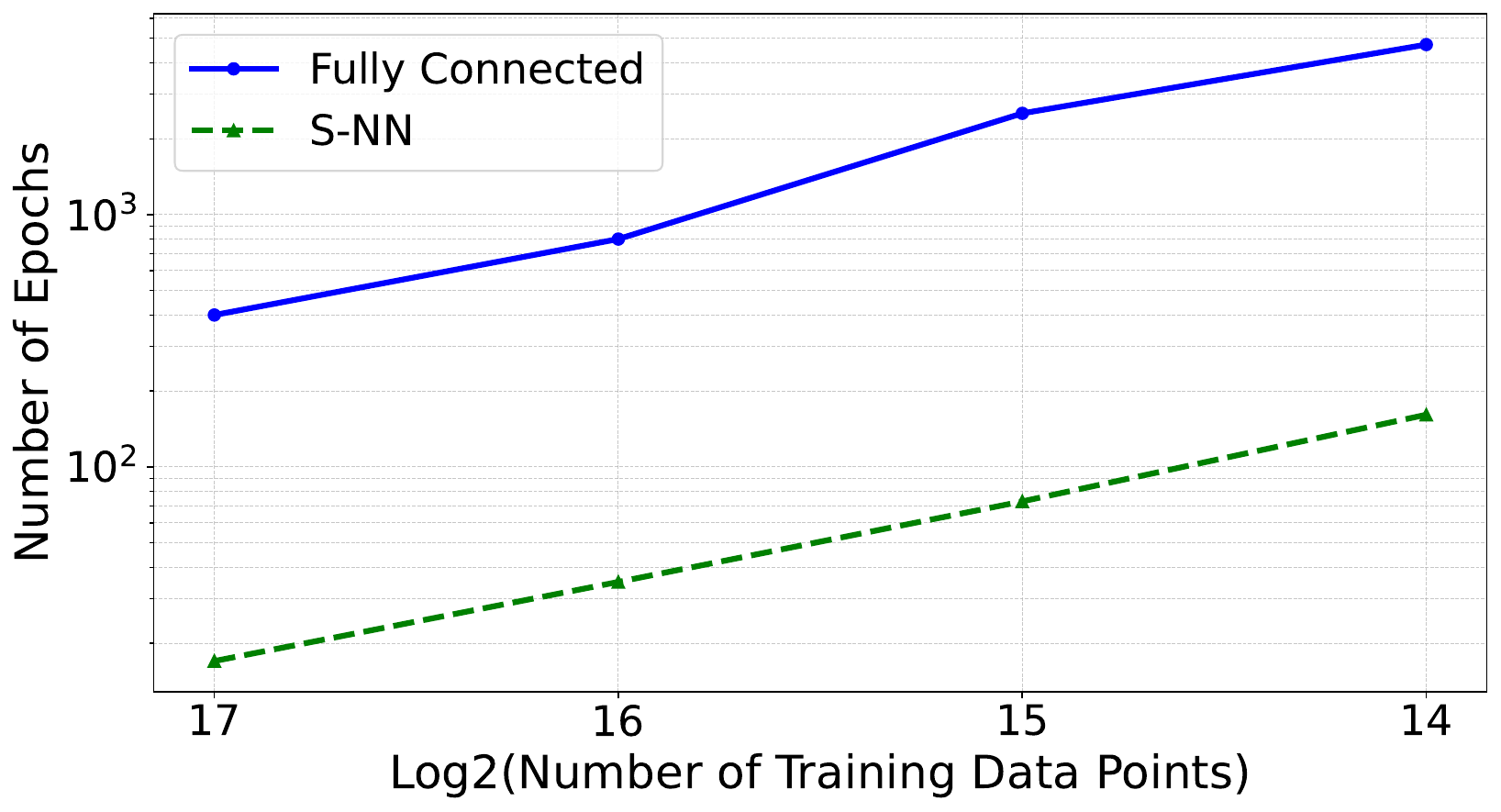}
    \caption{Comparison of training epochs (log-scale) against the number of training points.}
    \label{fig:min_epochs}
    \end{minipage}
\end{figure}

Furthermore, we compute approximations of the optimal value function $V$ for linear quadratic optimal control problems with increasing state dimensions, ranging from $20$ to $120$ in increments of $10$. Specifically, we considered two different cases: one where the matrix $A$ has bandwidth $1$, and another where $A$ has bandwidth $3$. We use a S-NN as shown in Figure \ref{fig:NN_architecture} with $n = s$ subnetworks and $\nu_g = \nu_z = 0.5$. The training dataset size was fixed at 4069 points, uniformly sampled in the unit ball, ensuring consistency with the framework established in Theorem \ref{thm.ExpDecay}. The graph-distance $l$ was fixed to $3$ in all cases. 
For each dimension, the size of the subsystems' layers was minimized while maintaining an MSE in the test data below \(10^{-4}\). It turned out that for both bandwidths and across all dimensions, only \(M = 2\) neurons per sublayer were required to satisfy this error threshold. Consequently, the total number of neurons and parameters increases linearly with the dimensionality of the problem, as illustrated in Figure \ref{fig:min_parameters}. 
These findings suggest that the polynomial bound established in the theorems in Section \ref{sec:sep} may not be sharp for LQR problems.
It is important to note that, while the graph for the case with bandwidth equal to 1 is sequential, as shown in Figure \ref{fig.NeighborhoodIllustration}, the increased number of connections in the case of bandwidth equal to 3 (depicted in Figure \ref{fig:bandwidth3}) leads to higher-dimensional graph neighborhoods. Specifically, for our choice of $l = 3$ we have $d = 4$ for the first case and $d = 10$ for the second. This, in turn, results in a larger number of trainable parameters for the setting with bandwidth equal to $3$, as shown in Figure \ref{fig:min_parameters}.  

\begin{figure}[h!]
    \centering
    \begin{minipage}[t]{0.45\textwidth}
       \centering
    \includegraphics[width=\textwidth]{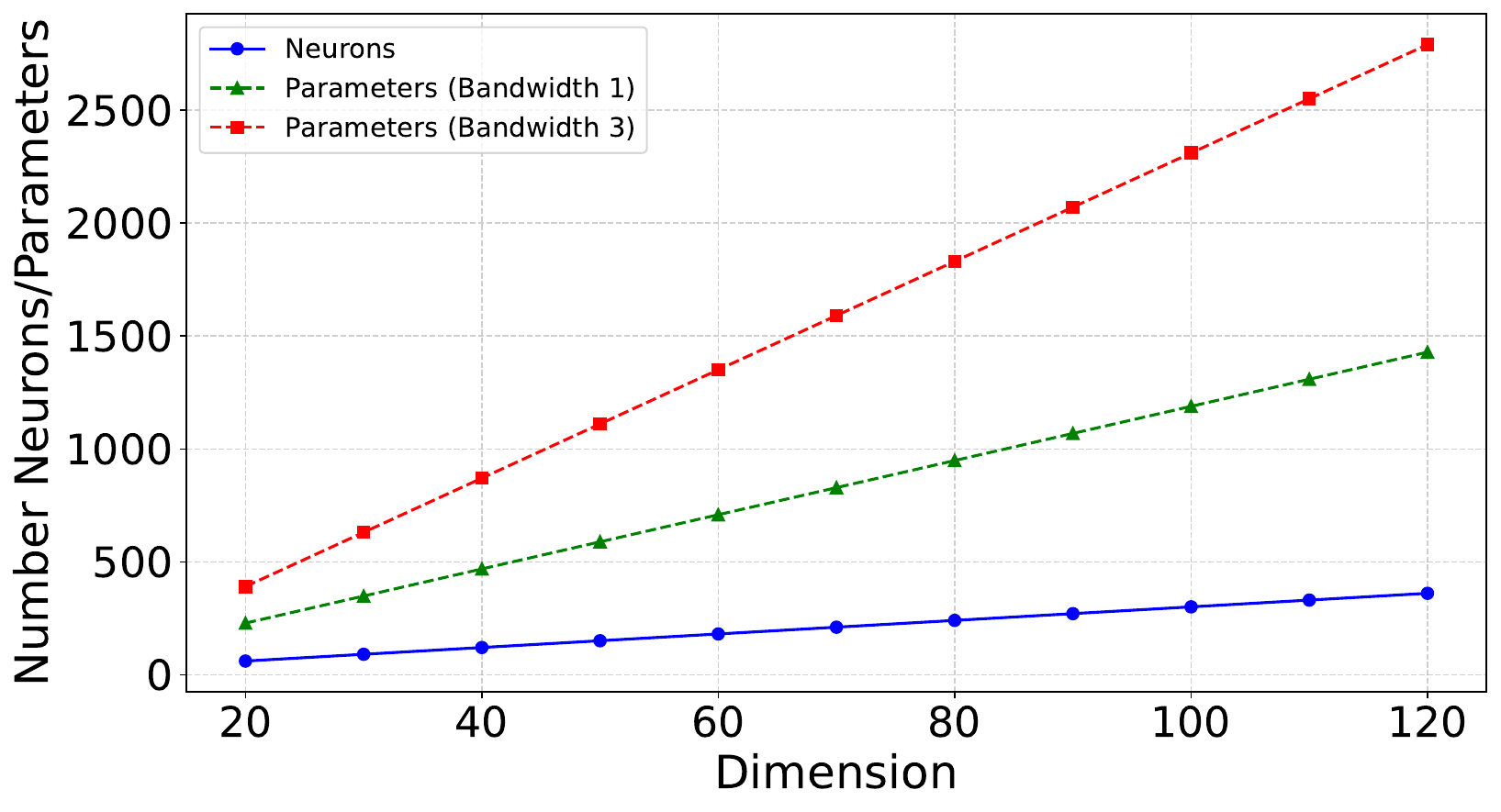}
    \caption{Number of neurons and parameters for bandwidths $1$ and $3$ in dependence of the dimension.}
    \label{fig:min_parameters}
    \end{minipage}%
    \hfill 
    \begin{minipage}[t]{0.45\textwidth}
    \centering
    \includegraphics[width=\linewidth]{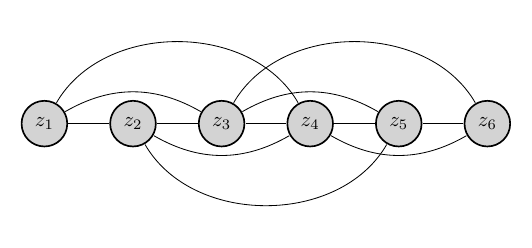}
    \caption{Graph corresponding to $A$ with bandwidth equal to $3$ and $B=Q=R = I_n$. }
    \label{fig:bandwidth3}
    \end{minipage}
\end{figure}

\begin{remark}
    While the experiments in this section involving the S-NN network structure exploit the decaying sensitivity in the $P$-matrix to approximate $V(x) = x^T P x$, we would like to point out that this decay does not imply a decay in the singular values of $P$. This contrast is illustrated in Figure~\ref{fig:SVD}, which shows the decay of the first column of $P$ (left) and the singular values of $P$ (right), for a linear quadratic problem with tridiagonal $A \in \mathbb{R}^{n \times n}$ and $B = Q = R = I_n$, with $n = 100$. We conclude that standard low-rank approximation techniques cannot exploit the decaying sensitivity property.
\end{remark}

\begin{figure}[h!]
    \centering
    \begin{minipage}[t]{0.45\textwidth}
    \centering
    \includegraphics[width=1\linewidth]{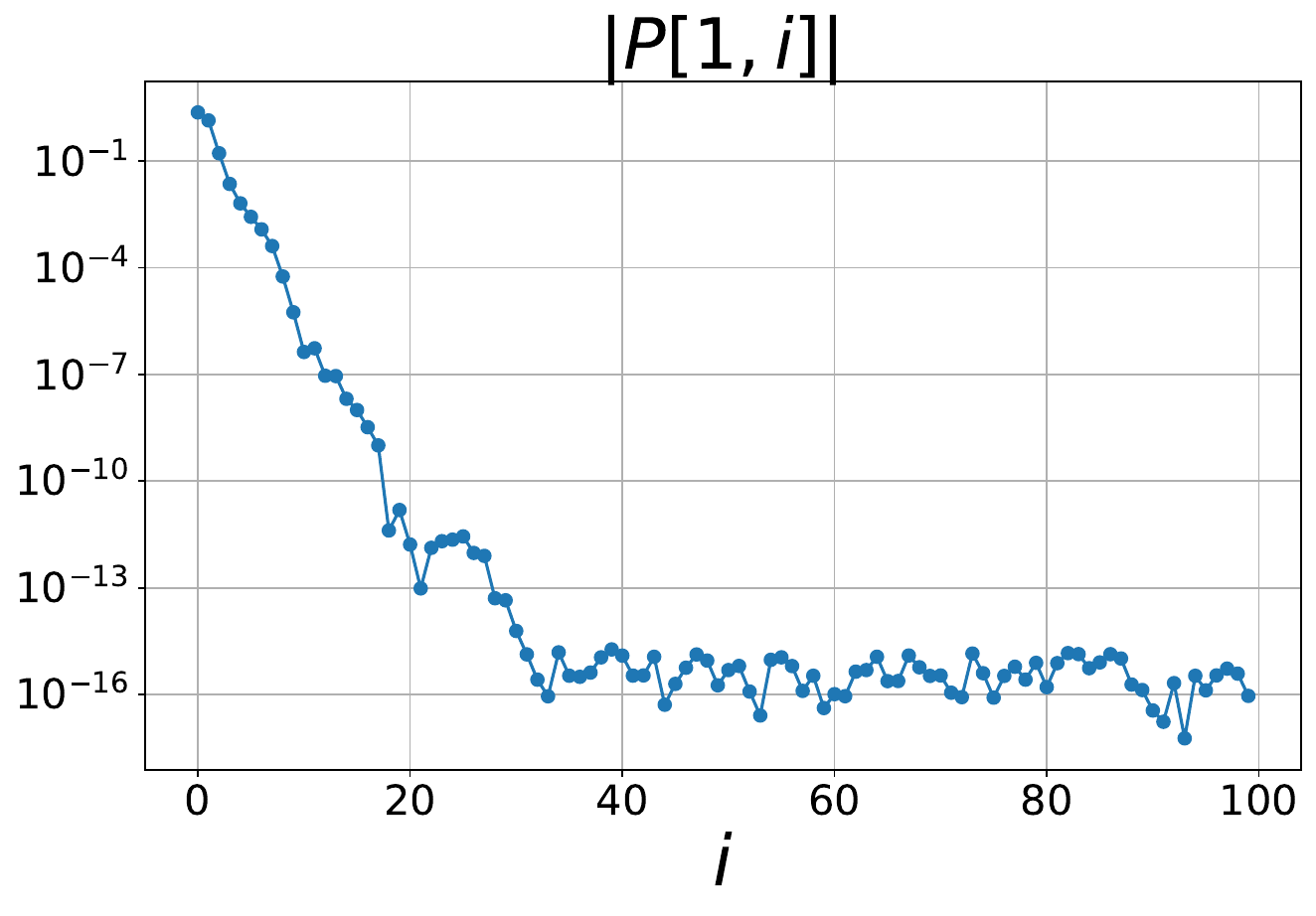}
    \end{minipage}%
    \hfill 
    \begin{minipage}[t]{0.45\textwidth}
    \centering
    \includegraphics[width=\linewidth]{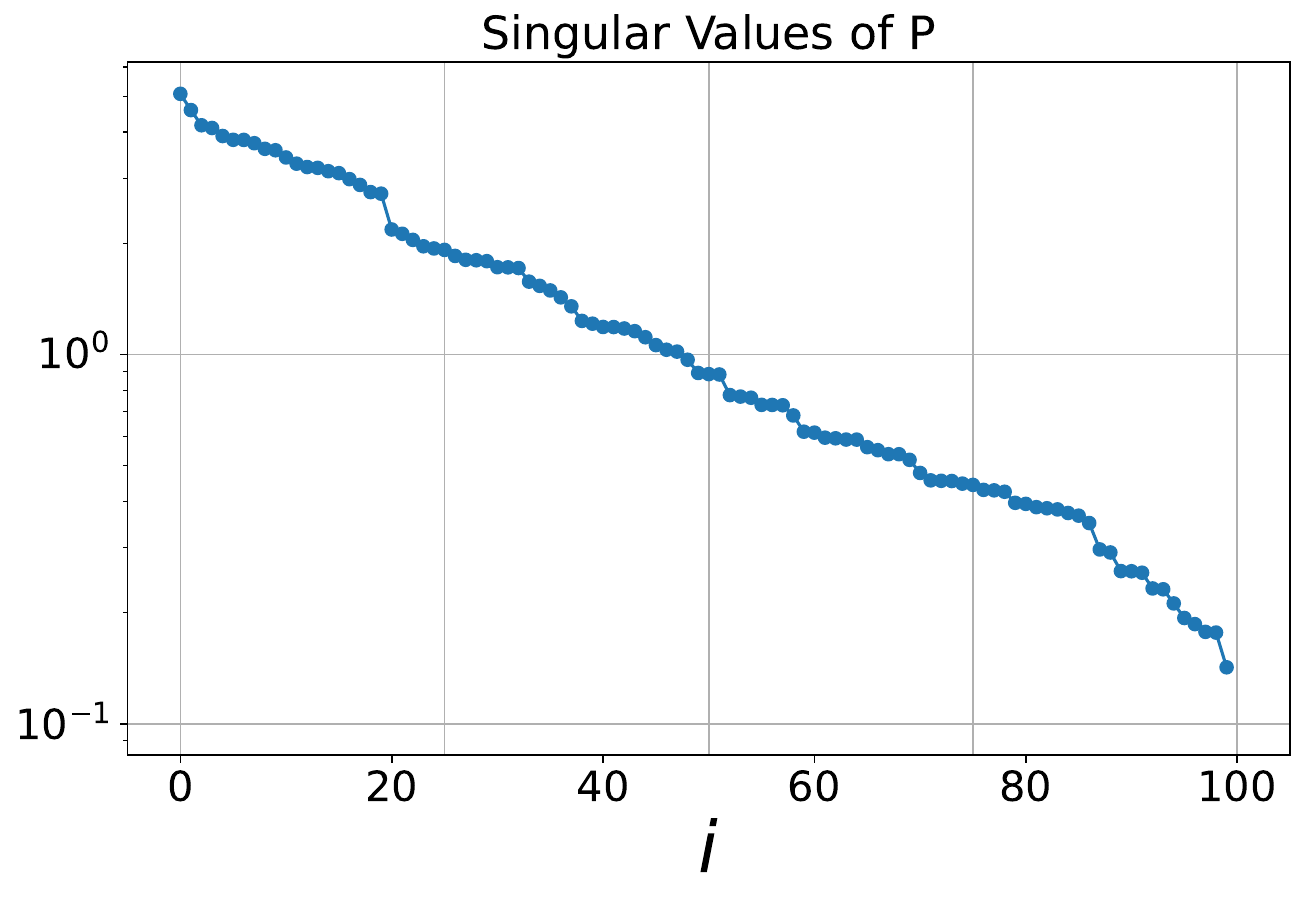}
    \end{minipage}

    \caption{Comparison of the decaying sensitivity with the singular values of $P$.}
    \label{fig:SVD}
\end{figure}

\subsection{Approximation of a non-quadratic function}
In this subsection, we briefly discuss an application of the training algorithm presented in Subsection \ref{subsec:Numerics_LQR} towards the approximation of non-quadratic functions. To this end, for $\rho \in (0,1]$ we introduce the function $V \colon [-1,1]^{100} \to \R$ defined as 
\begin{equation*}
    V(x) = \sum_{i, j=1}^n \sin(x_i) \sin(x_j) \rho^{\abs{i-j}},
\end{equation*}
and consider a graph $\mathcal{G}$ with nodes $\lbrace 1, \dots, n \rbrace$ and sequential connections leading to $\distG(i,j) = \abs{i-j}$ for $i,j \in \lbrace 1, \dots, n \rbrace$. Note that the function $V$ is not separable with respect to Definition \ref{def:separable} since each entry of its Hessian matrix is non-zero. However, for $\rho < 1$ it exhibits an exponentially decaying sensitivity with basis $\rho$. We employ a S-NN with a distance for the graph-neighborhoods of $l=5$, a sublayer-size of $M=16$ in each subnetwork, $\nu_g = \nu_z = 0.5$, and $8192$ training, validation, and test data. The resulting MSE after training for $10^4$ epochs is shown in Figure \ref{fig:error_vs_rho}, where we observe that the function is not well approximated for  $\rho = 1$, which corresponds to the case without decaying sensitivity. In contrast, for the case with exponentially decaying sensitivity, where $\rho = \frac{1}{8}$, the MSE falls below $10^{-4}$. \\ 
A similar behavior can be observed if we replace the exponential decaying sensitivity of $V$ with a polynomial decay. That is, we consider the function $\widetilde{V} \colon [-1,1]^{100} \to \R$ given as 
\begin{equation*}
    \widetilde{V}(x) = \sum_{i, j=1}^n \sin(x_i) \sin(x_j) (\abs{i-j}+1)^{-\alpha},
\end{equation*}
for different values of $\alpha \in \{ 0, 1, 2, 3\}$.  We use the same network architecture and training parameters as for $V$ and observe that the stronger the polynomial degree is, the better the resulting network approximation is, see Figure \ref{fig:error_vs_alpha}. 

It follows that the separable structured neural network facilitates the approximation of general functions, with the degree of effectiveness being influenced by the rate of decay of the sensitivity. Specifically, the network architecture improves the ability to capture functions with decaying sensitivity, with the approximation quality being directly related to the rate at which the sensitivity diminishes. As the sensitivity decays more rapidly, the separable structure becomes increasingly advantageous, enabling more accurate and efficient representations of the function in question.

\begin{figure}[h!]
    \centering
    \begin{minipage}[t]{0.45\textwidth}
    \centering
    \includegraphics[width=1\linewidth]{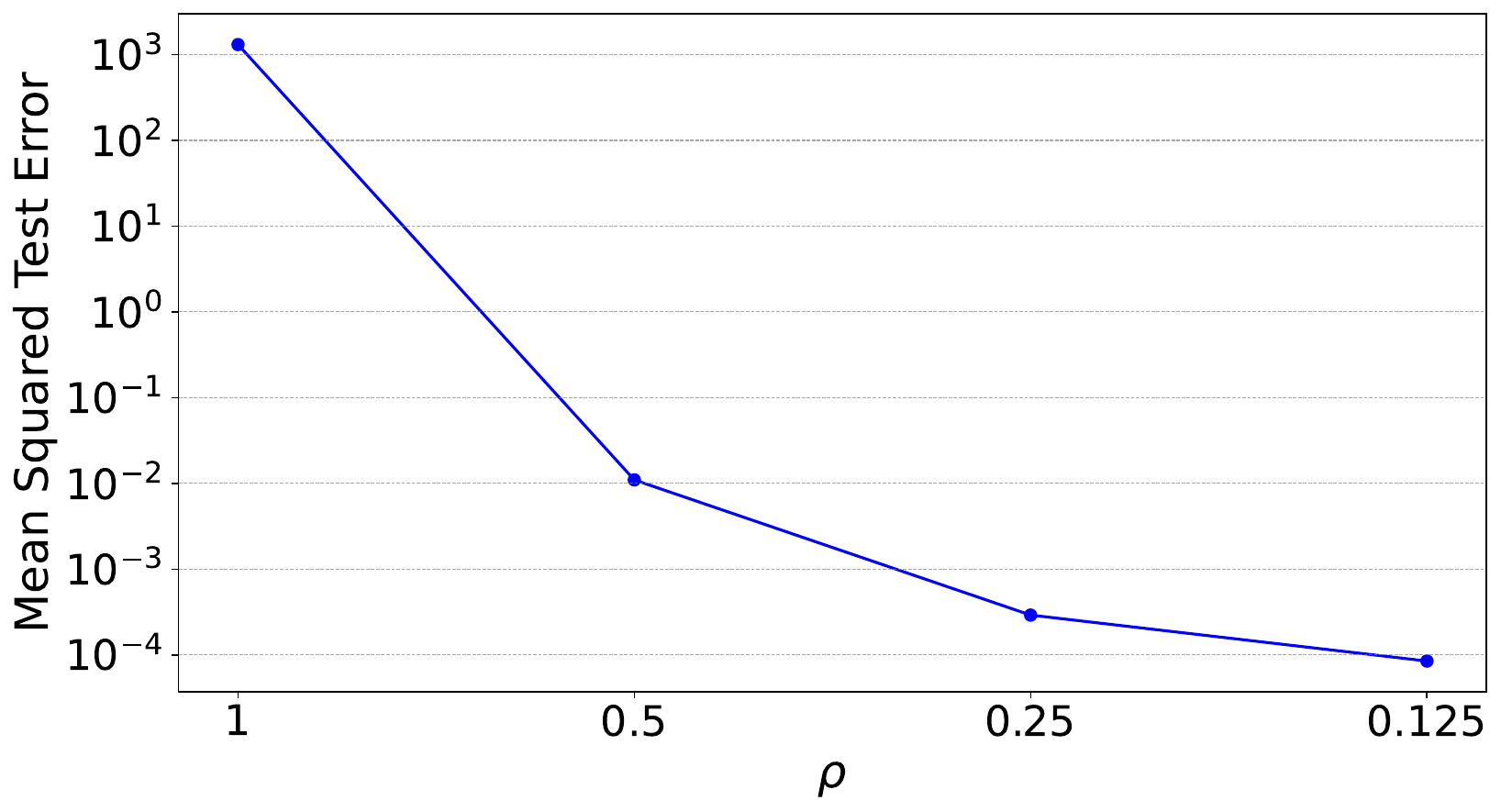}
    \caption{Test error (log-scale) for approximating $V$ (exponential decay) in dependence of $\rho$.}
    \label{fig:error_vs_rho}
    \end{minipage}%
    \hfill 
    \begin{minipage}[t]{0.45\textwidth}
    \centering
    \includegraphics[width=\linewidth]{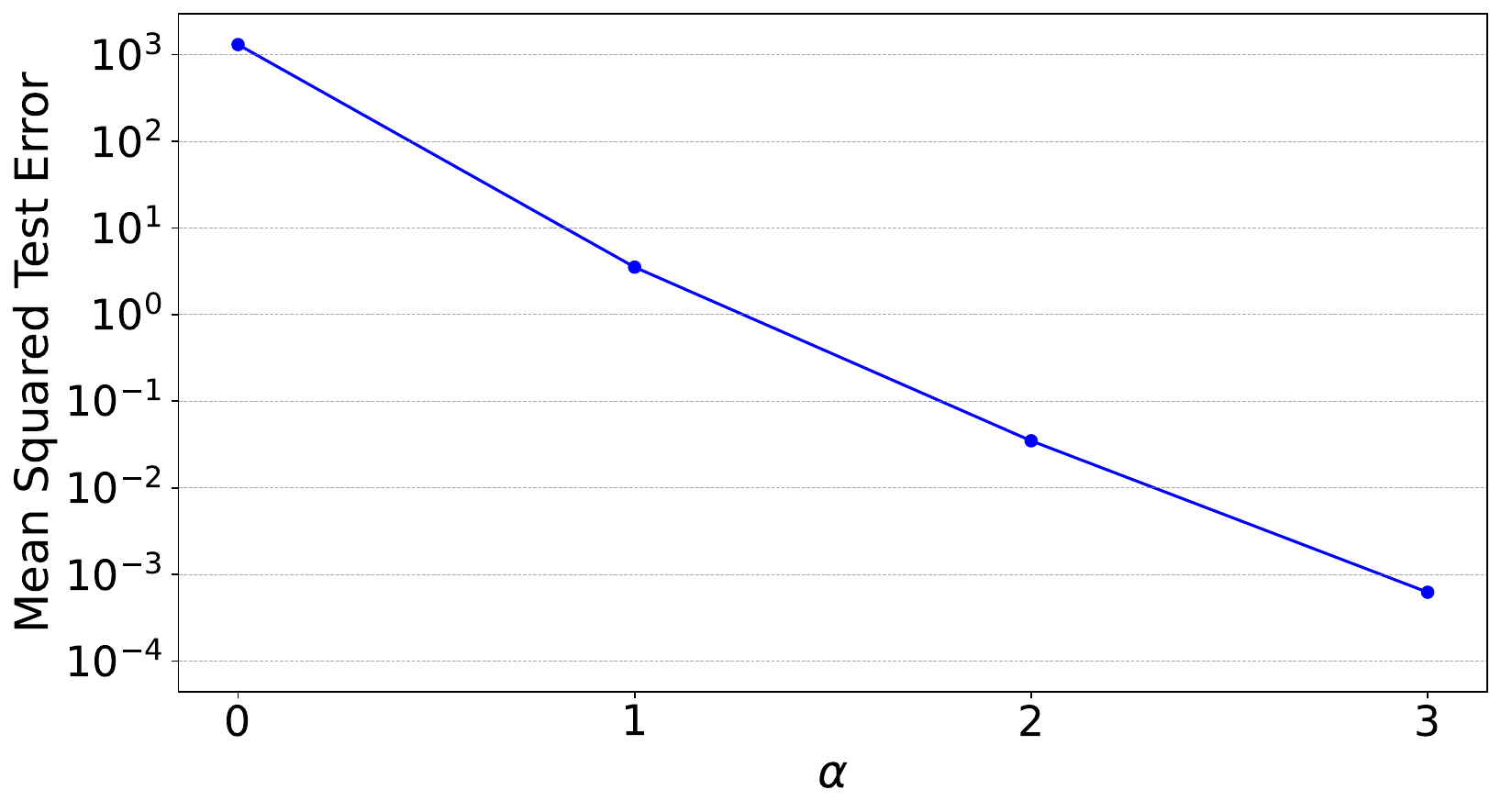}
    \caption{Test error (log-scale) for approximating $\widetilde{V}$ (polynomial decay) in dependence of $\alpha$.}
    \label{fig:error_vs_alpha}
    \end{minipage}
\end{figure}

\subsection{Nonlinear optimal control problem}

Finally, we turn our attention to a more challenging setting: the approximation of optimal value functions for systems governed by nonlinear dynamics. As a representative example, we consider the infinite-horizon optimal control problem associated with the Allen--Cahn equation, a nonlinear parabolic PDE of reaction-diffusion type, given by
\begin{equation*}
\left\{
\begin{aligned}
\partial_t y(t,x) &= \sigma\, \partial_{xx} y(t,x) + y(t,x)(1 - y(t,x)^2) + u(t,x), \\
y(0,x) &= y_0(x),
\end{aligned}
\right.
\end{equation*}
on the spatial domain \( x \in [0,1] \), equipped with homogeneous Neumann boundary conditions. The goal is to regulate the system toward the unstable steady state \( \tilde{y}(x) \equiv 0 \) by minimizing the quadratic cost
\[
\tilde{J}(y_0, u) = \int_0^\infty \int_0^1 \left( \delta_1 |y(t,x)|^2 + \delta_2|u(t,x)|^2 \right) dx \, dt, \quad \delta_1,\delta_2 >0.
\]
The optimal value function for this nonlinear problem is not available in closed form. In order to be able to assess the performance of the separable neural network architecture for this example, we thus learn an approximation of the optimal value function to which we can then compare the function represented by the network. For obtaining this approximation, we discretize the state in space using finite differences with \( n = 50 \) grid points, yielding a finite-dimensional nonlinear optimal control problem of the form
\begin{align} 
    \begin{split} \label{eq.AllenCahn}
        \min_{u \in \mathcal{U}_c} J(u, y_0) & =  \int_0^\infty  y(t)^\top Q y(t) +  u(t)^\top R u(t) \, dt,\\
        \dot{y}(t) & = A(y(t)) y(t) + B u(t),
    \end{split}
\end{align}
where \( y(t) \in \mathbb{R}^n \), \( A(y) \in \mathbb{R}^{n \times n} \) depends nonlinearly on the state, \( B = I_n\), $Q = \frac{\delta_1}{n} I_n$, and $R= \frac{\delta_2}{n} I_n $. We approximate the optimal value function \( V \) of \eqref{eq.AllenCahn} by a quadratic state-dependent ansatz: We set $V(y) = y^\top P(y) y $, 
where \( P(y) \) is obtained as the solution to the State-Dependent Riccati Equation (SDRE):
\begin{equation}
A^\top(y) P(y) + P(y) A(y) - P(y) B R^{-1} B^\top P(y) + Q = 0.
\label{SDRE}
\end{equation}

This construction does not solve the associated Hamilton--Jacobi--Bellman equation exactly, but provides a pointwise approximation of the value function by enforcing local optimality at each state \( y \). The SDRE approach constitutes a generalization of the classical LQR scheme, designed to incorporate the state-dependent nonlinearities present in the system dynamics. or a detailed discussion of the theoretical underpinnings, we refer to \cite{ccimen2008state} and the references therein, while applications in the context of optimal control for PDEs can be found in \cite{albi2022,alla2023state}.
Note that while $A(y) \in \R^{n \times n}$ depends on the state in this non-linear setting, its non-zero entries are independent of $y$. Thus, we can define the graph corresponding to $V$ with subsystems $z_j = y_j$, $1 \leq j \leq n$, analogously to the linear quadratic setting described in Section \ref{sec.ExistenceDecay} leading to a problem formulation fitting Setting \ref{setting:general} with $\distG(i,j) = | i - j |$. 

We now examine the emergence of a decaying sensitivity property in this nonlinear $50$-dimensional setting. By Definition \ref{def:setting:sensitivity} this requires to investigate the Lipschitz properties of the functions $g_j$ defined in \eqref{eq:setting:def_gj}. To this end, we fix the parameters
$\sigma = 10^{-2}$, $\delta_{1} = 10$, $\delta_{2} = 0.1$, and consider a fixed reference point $x \in \Rn$. For each index \(i \in \{1,\ldots,n\}\), we define a perturbed vector \(\tilde{x}_i \in \mathbb{R}^n\) by
\[
\tilde{x}_{i,k} = 
\begin{cases}
0, & \text{if } k = i, \\
x_k, & \text{otherwise}.
\end{cases}
\]
A numerical estimate of the local Lipschitz constant of \(g_j\) with respect to the \(i\)th coordinate is then given by
\[
\tilde{L}_{i,j}
= 
\frac{\lvert g_j(x) - g_j(\tilde{x}_i) \rvert}
{\|x - \tilde{x}_i\|}.
\]
The right panels of Figure~\ref{fig:V_50} report the empirical estimates \(\tilde{L}_{i,j}\) for \(j=1\) (top) and \(j=25\) (bottom). Each estimate is computed at two reference points obtained by setting the \(j\)-th component to \(\alpha\) and all other components to \(0.1\), with \(\alpha\in\{0.1,0.01\}\). Results are shown on a logarithmic scale. The observed exponential decay of $\tilde{L}_{i,j}$ with increasing $|i - j|$, as well as its decrease for smaller $\|z_j\|$, provide strong numerical evidence that $V$ satisfies the decaying sensitivity property as defined in Definition~\ref{def:setting:sensitivity}.

This structural property justifies the use of the separable neural network architecture presented in Section \ref{sec:NN} to leverage localized interactions to yield an accurate and efficient approximation of \(V\). We employ a S-NN architecture configured with a graph‑neighborhood radius of \(l = 3\) and \(M = 32\) neurons in each subnetwork, resulting in a total of \(9409\) trainable parameters. The dataset comprises \(1024\) samples for training, validation, and testing uniformly randomly chosen in $\Omega = \{ x \in \R^{50} \mid \| x \|_2 \le 1 \}$, respectively.  As regularization weights we chose \(\nu_{g} = 10^{4}\) and \(\nu_{z} = 10\). The large value of \(\nu_{g}\) discourages flat regions in the surrogate, while the stringent tolerance guarantees a close fit to the reference solution; a systematic study of these effects is left for future work. 

The left and middle panels in Figure~\ref{fig:V_50} shows the results of a training process conducted with a stopping tolerance of \(\epsilon = 10^{-7}\). Note that while training was carried out on $\Omega$, i.e., the unit ball in $\R^{50}$, the plots are generated for visualization purposes on two-dimensional slices of the hypercube $[-1,1]^{50}$. We depict two representative coordinate planes, namely \((x_1, x_2)\) (top) and \((x_{25}, x_{50})\) (bottom). The left panel depicts the true solution derived from the SDRE ansatz, while the middle panel shows the corresponding approximation generated by the S-NN. The visual agreement between the reference and surrogate value functions is strong, with the SNN model achieving a mean squared error in the test data of approximately \(8.32 \times 10^{-8}\) after 210 training epochs. 

\begin{figure}[ht]
    \centering
    \begin{subfigure}{0.32\textwidth}
        \centering
        \includegraphics[width=\textwidth]{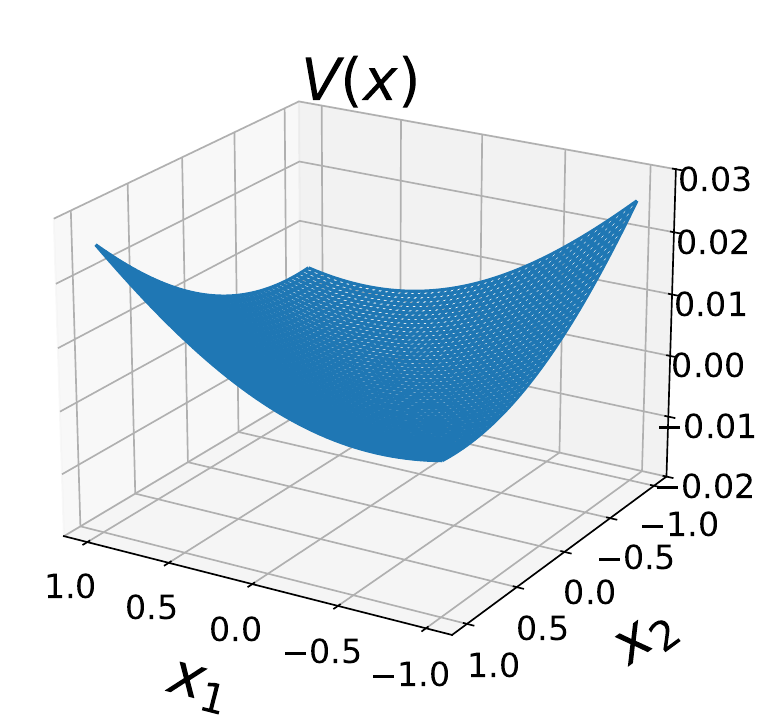}
    \end{subfigure}
    \hfill
    \begin{subfigure}{0.32\textwidth}
        \centering
        \includegraphics[width=\textwidth]{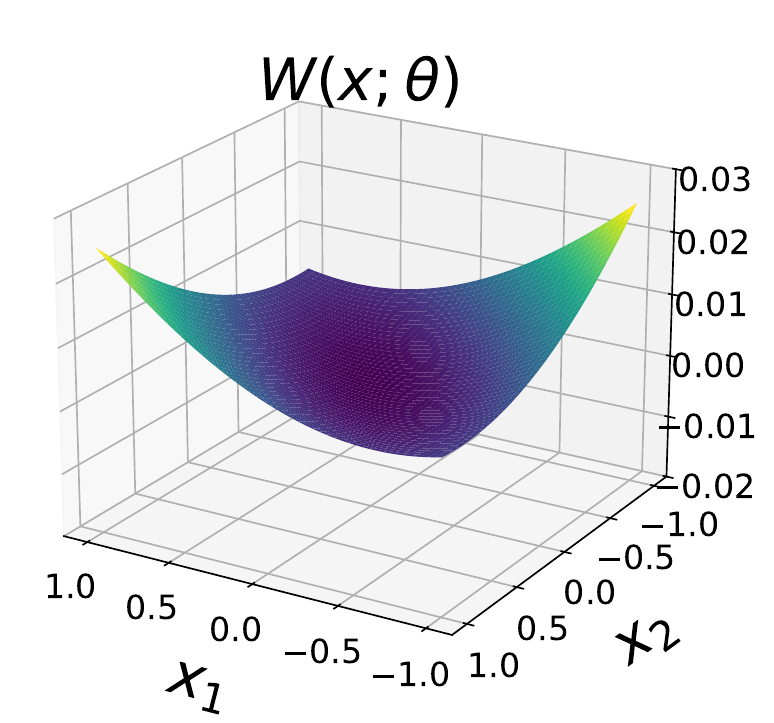}
    \end{subfigure}
    \hfill
    \begin{subfigure}{0.32\textwidth}
        \centering
        \includegraphics[width=\textwidth]{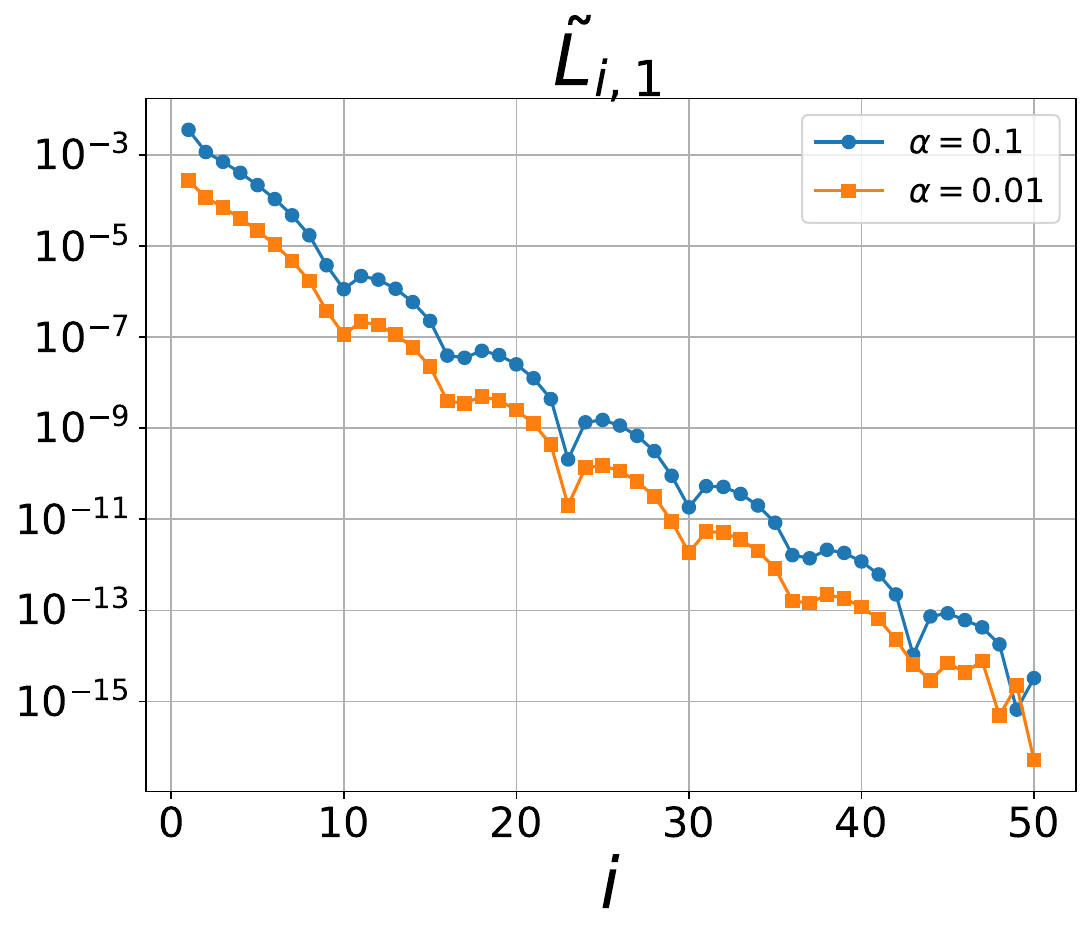}
    \end{subfigure}
    
    \vspace{0.5em} 

    \begin{subfigure}{0.32\textwidth}
        \centering
        \includegraphics[width=\textwidth]{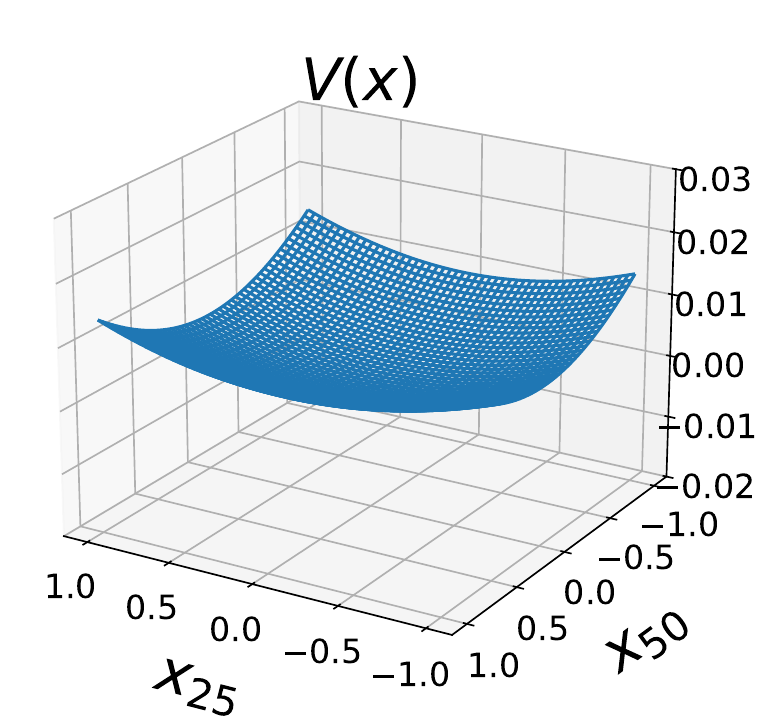}
    \end{subfigure}
    \hfill
    \begin{subfigure}{0.32\textwidth}
        \centering
        \includegraphics[width=\textwidth]{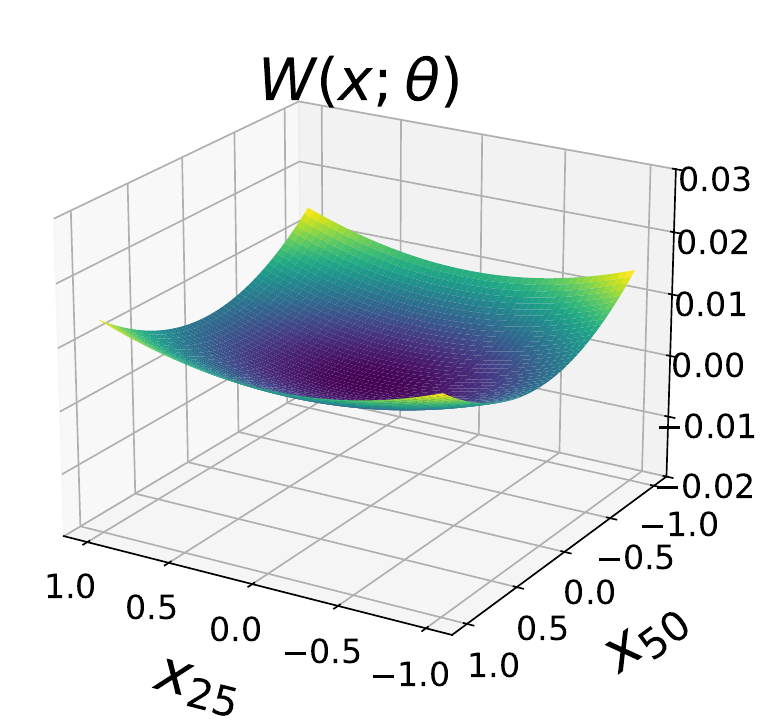}
    \end{subfigure}
    \hfill
    \begin{subfigure}{0.32\textwidth}
        \centering
        \includegraphics[width=\textwidth]{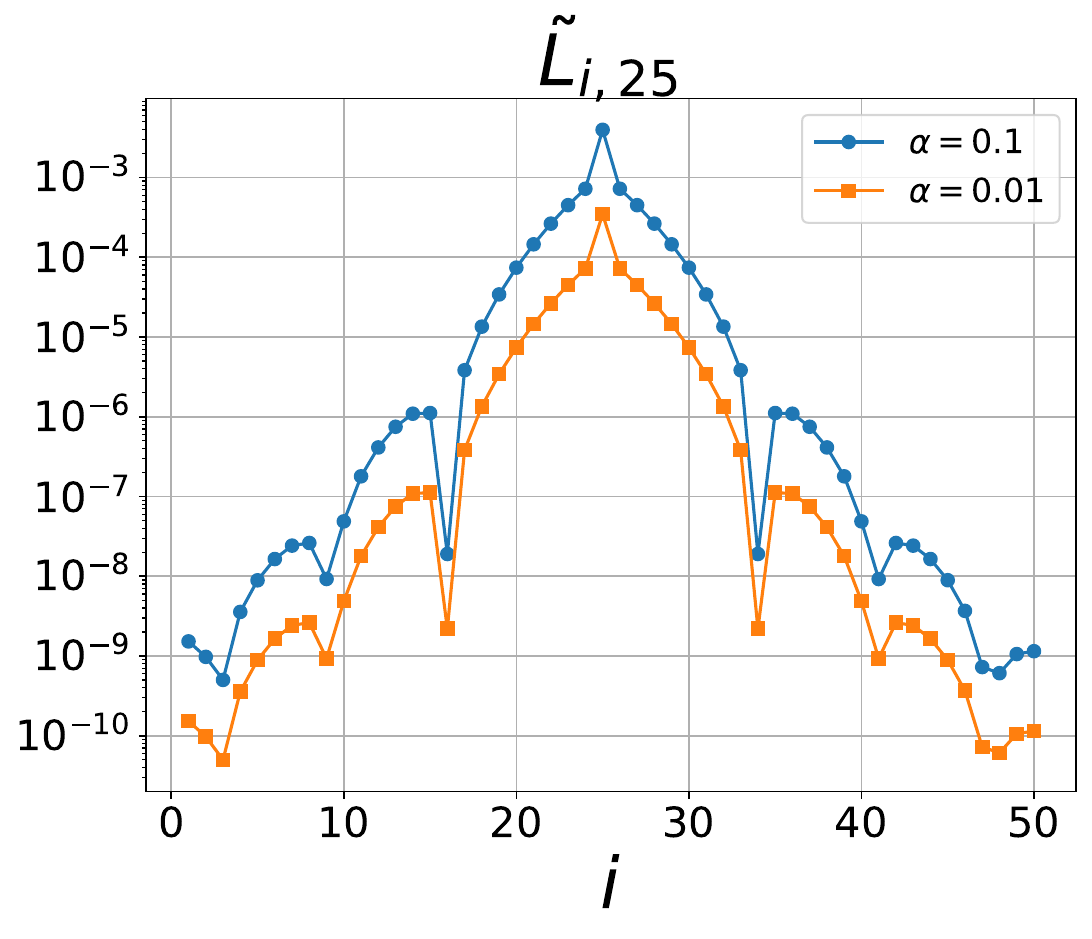}
    \end{subfigure}

    \caption{Plot of $V$ projected onto the $(x_1, x_2)$-axis (upper left) and onto the $(x_{25}, x_{50})$-axis (lower left), the S-NN approximation onto the same hyperplanes (middle column), and the decay of $\tilde{L}_{i,j}$ for $j=1$ (upper right) and $j=25$ (lower right) varying the index $i$.}
    \label{fig:V_50}
\end{figure}

\section{Conclusion}

This work has presented a comprehensive framework for constructing separable approximations of high-dimensional functions, emphasizing their representation through neural networks. The key contribution lies in leveraging the property of decaying sensitivity, which facilitates the decomposition of complex functions into localized components with reduced computational complexity. We demonstrated that separable approximations are not only theoretically viable but also practical in the context of optimal control problems, particularly in the LQR setting. To extend this approach to a nonlinear setting, datasets for the optimal control problem can be systematically generated utilizing either Pontryagin's Maximum Principle \cite{azmi2021optimal} or the State-Dependent Riccati Equation \cite{albi2022,dolgov2022data}. Moreover, the integration of separable structures with neural networks underscores the potential for addressing the curse of dimensionality in a wide range of applications. The resulting architectures allow for efficient representation and computation of value functions.

Future research directions include extending the existence results on decaying sensitivity and the proposed numerical methods to multi-agent control problems, exploring adaptive strategies for graph-based sensitivity modeling, and further optimizing neural network architectures to enhance generalization and computational efficiency.

{\bibliographystyle{abbrv}
\bibliography{literature_Sensitivity} 
}

\end{document}